\documentclass[graybox]{svmult}

\usepackage{type1cm}        
\usepackage{makeidx}         
\usepackage{graphicx}        
\usepackage{multicol}        
\usepackage[bottom]{footmisc}

\usepackage{newtxtext}       %
\usepackage{newtxmath}       

\makeindex             

\usepackage{tikz-cd}

\begin{document}

\title*{Gelfand triples for the Kohn-Nirenberg quantization on homogeneous Lie groups}
\titlerunning{Gelfand triples for homogeneous Lie groups}
\author{Jonas Brinker \and Jens Wirth}
\institute{Jonas Brinker \at Institut f\"ur Analysis, Dynamik und Modellierung, Universit\"at Stuttgart, Pfaffenwaldring 57, 70569 Stuttgart, \email{jonas.brinker@mathematik.uni-stuttgart.de}
\and
Jens Wirth\at Institut f\"ur Analysis, Dynamik und Modellierung, Universit\"at Stuttgart, Pfaffenwaldring 57, 70569 Stuttgart, \email{jens.wirth@mathematik.uni-stuttgart.de}}
%
%
\maketitle

\abstract*{We study the action of the group Fourier transform and of the Kohn-Nirenberg quantization \cite{ruzhansky_nilpotent} on certain Gelfand triples for homogeneous Lie groups $G$. Even for the Heisenberg group $G=\mathbb H$ there seems to be no simple intrinsic characterization for the Fourier image of the Schwartz space of rapidly decreasing smooth functions $\mathcal S(G)$, see \cite{GELLERschwartzfunctions,Astengo_Fouriertransform_on_Schwartzspace}. But we may derive a simple characterization of the Fourier image for a certain subspace $\mathcal S_*(G)$ of $\mathcal S(G)$. We restrict our considerations to the case, where $G$ admits irreducible unitary representations, that are square integrable modulo the center $Z(G)$ of $G$, and where $\dim  Z(G)=1$. This enables us to use an especially applicable characterization of these irreducible unitary representations that are square integrable modulo $Z(G)$ \cite{Moore_Wolf,ruzhansky_flat_orbit_quantization,GROCHENIG_rottensteiner_bases}. Also, Pedersen's machinery \cite{pedersen_matrix_coefficients} combines very well with this setting \cite{ruzhansky_flat_orbit_quantization}. Starting with $\mathcal S_*(G)$, we are able to construct distributions and Gelfand triples around $L^2(G,\mu)$ and its Fourier image $L^2(\widehat G,\widehat \mu)$, such that the Fourier transform becomes a Gelfand triple isomorphism. In this context we show for the Fourier side, that multiplication of distributions with a large class of vector valued smooth functions is possible and well behaved.
Furthermore, we rewrite the Kohn-Nirenberg quantization as an isomorphism for our new Gelfand triples and prove a formula for the Kohn-Nirenberg symbol, which is known from the compact group case \cite{turunen_ruzhansky_psydos_and_symmetries}.
}

\abstract{We study the action of the group Fourier transform and of the Kohn-Nirenberg quantization \cite{ruzhansky_nilpotent} on certain Gelfand triples for homogeneous Lie groups $G$. Even for the Heisenberg group $G=\mathbb H$ there seems to be no simple intrinsic characterization for the Fourier image of the Schwartz space of rapidly decreasing smooth functions $\mathcal S(G)$, see \cite{GELLERschwartzfunctions,Astengo_Fouriertransform_on_Schwartzspace}. But we may derive a simple characterization of the Fourier image for a certain subspace $\mathcal S_*(G)$ of $\mathcal S(G)$. We restrict our considerations to the case, where $G$ admits irreducible unitary representations, that are square integrable modulo the center $Z(G)$ of $G$, and where $\dim  Z(G)=1$. This enables us to use an especially applicable characterization of these irreducible unitary representations that are square integrable modulo $Z(G)$ \cite{Moore_Wolf,ruzhansky_flat_orbit_quantization,GROCHENIG_rottensteiner_bases}. Also, Pedersen's machinery \cite{pedersen_matrix_coefficients} combines very well with this setting \cite{ruzhansky_flat_orbit_quantization}. Starting with $\mathcal S_*(G)$, we are able to construct distributions and Gelfand triples around $L^2(G,\mu)$ and its Fourier image $L^2(\widehat G,\widehat \mu)$, such that the Fourier transform becomes a Gelfand triple isomorphism. In this context we show for the Fourier side, that multiplication of distributions with a large class of vector valued smooth functions is possible and well behaved.
Furthermore, we rewrite the Kohn-Nirenberg quantization as an isomorphism for our new Gelfand triples and prove a formula for the Kohn-Nirenberg symbol, which is known from the compact group case \cite{turunen_ruzhansky_psydos_and_symmetries}.
}


\section{Introduction}

There is an imbalance between the Schwartz space of rapidly decreasing smooth functions $\mathcal S(G)$ and its Fourier image $\mathcal S(\widehat G)$ for simply connected nilpotent Lie groups $G$. For the space $\mathcal S(\widehat G)$ is a space of operator valued functions on the irreducible unitary representations of $G$, that is not easily characterized without relying on the Fourier transform itself, see \cite{GELLERschwartzfunctions,Astengo_Fouriertransform_on_Schwartzspace}. Of course we can see $\mathcal S(G)$ as a projective limit of suitable Hilbert spaces, which results in a corresponding representation of $\mathcal S(\widehat G)$ as a projective limit. But this approach is rather cumbersome, if we want to identify multiplication operators on $\mathcal S(\widehat G)$. Though, if $G$ is also a homogeneous group with one-dimensional center, then the dual $\widehat G$ can be parametrized (up to a null set of the Plancherel measure) by $\mathbb{R} \setminus\{0\}= \mathbb{R}^\times$. The setting of homogeneous Lie groups with square integrable (modulo the kernel) representations seems to be convenient in general, see e.g. \cite{ruzhansky_flat_orbit_quantization,GROCHENIG_rottensteiner_bases}. Now the main idea is to define a subspace $\mathcal S_*(G)$ of $\mathcal S(G)$, such that its Fourier image can be identified with a tensor product of rapidly decreasing functions on $\mathbb{R}^\times$ and a well behaved space of operators isomorphic to $\mathcal L(\mathcal S'(\mathbb{R}^n),\mathcal S(\mathbb{R}^n))$. Here we choose $\mathcal S_*(G)$, such that $\mathcal S_*(G)$ is still densely embedded into $L^2(G,\mu)$. Hence, we may construct a Gelfand triple $\mathcal G_*(G)$ from $\mathcal S_*(G)$ and $L^2(G,\mu)$ and show that the Fourier transform acts as a Gelfand triple isomorphism. By using the theory of vector valued distributions of L. Schwartz and related results in \cite{bargetz_char_of_mult_spaces}, we may define multiplication operators on the Fourier side. We will employ the concept of polynomial manifolds, which were used in \cite{pedersen_geometric_quant}, for the corresponding spaces $\mathcal S(\mathbb{R}^\times)$ and $\mathcal O_\mathrm{M}(\mathbb{R}^\times)$ of smooth functions on $\mathbb{R}^\times$ with growth conditions.

Using the Fourier transform on $\mathcal S_*(G)$, we will also examine the Kohn-Nirenberg quantization, defined in \cite{ruzhansky_nilpotent}. We incorporate our new function spaces into Gelfand triples of symbols and operators, on which the Kohn-Nirenberg quantization acts as a Gelfand triple isomorphism. Finally, we will prove a formula for the Kohn-Nirenberg symbol of operators $A\in\mathcal L(\mathcal O_\mathrm{M}(G))$, that is motivated by the corresponding formula $a(x,\xi) = \xi^*\cdot A(\xi)$ for symbols of operators $A\in\mathcal L(\mathcal D(H))$ for compact Lie groups $H$ from \cite{turunen_ruzhansky_psydos_and_symmetries}.

The paper is structured as follows: Subsections 1.1 and 1.2 are dedicated to recalling common facts about harmonic analysis and functional analysis. At the end of subsection 1.2 we cite theorems about the continuation of bilinear maps to completions of topological tensor products, which will be of importance for the multiplication of vector valued functions and distributions. In section 1.3 we define Gelfand triples with additional real structure and discuss direct sums, tensor products and kernel maps in this context.

Section 2 is dedicated to the definition of new Gelfand triples for the Fourier transform, by using polynomial manifolds and Pedersen's quantization procedure. We start by recalling basic concepts from the theory of vector valued functions and by defining polynomial manifolds. Then, in subsection 2.1, we pay special attention to the polynomial manifold $\mathbb R^\times$, the space of Schwartz functions $\mathcal S(\mathbb R^\times)$ and the dual space of $\mathcal S(\mathbb R^\times)\,\hat\otimes\, E$. In the succeeding subsection we recall the characterization of irreducible representations by the orbit method and Pedersen's machinery and apply both to homogenous Lie groups $G$ with irreducible representations, that are square integrable modulo the center $Z(G)$. In subsection 2.3, using the methods developed in the previous two subsections, we define an adjusted group Fourier transform $\mathcal F_\pi$. Finally, in subsection 2.4, we define the test function space $\mathcal S_*(G)$, define the corresponding Gelfand triple $\mathcal G_*(G)$ and characterize the Gelfand triple $\mathcal G(\mathbb R^\times;\pi)$, that is isomorphic to $\mathcal G_*(G)$ via the adjusted Fourier transform. We finish this section by discussing multiplication operators on $\mathcal S(\mathbb R^\times;\pi)$, the Fourier side of $\mathcal S_*(G)$.

In the last section we show, in what way the Kohn-Nirenberg quantization extends to a Gelfand triple isomorphism from the space of operator valued test functions $\mathcal S(G)\,\hat\otimes\,\mathcal S(\mathbb R^\times;\pi)$. Finally in subsection 3.2 we prove the formula for the Kohn-Nirenberg symbol of an operator $A\in\mathcal L(\mathcal O_\mathrm{M}(G))$.

Now we start by reminding ourselves of some standard notations and concepts.

\subsection{Generalities}

By $G$ we will always denote a simply connected nilpotent lie group. Since $G$ is diffeomorphic to its Lie algebra $\mathfrak g$ via the exponential map, we will model $G$ to be $\mathfrak g$ as set. I.e. $G=\mathfrak{g}$ is a Lie algebra and a group with multiplication $(x,y)\mapsto xy$ given by the Baker-Campbell-Hausdorff formula. We will use the symbol $G$ or $\mathfrak g$ depending on which property we want to emphasise. The center of $G=\mathfrak g$ will be denoted by $Z(G)=\mathfrak{z}$. We will denote the space of left invariant differential operators on $G$ by $\mathfrak{u}(\mathfrak{g}_\mathrm{L})$.
We choose a fixed Haar measure $\mu$ on the group $\mathfrak g = G$. This measure $\mu$ is both a Haar measure with respect to the multiplication and addition. Note, that for each $P\in\mathfrak{u}(\mathfrak{g}_\mathrm{L})$, there is a unique left invariant differential operator $P^\mathrm{t}$, such that
\begin{equation*}
	\int_G (P \varphi)\,\psi\,\mathrm{d}\mu = \int_G\varphi\, P^\mathrm{t}\psi\,\mathrm{d}\mu,\quad\text{for all }\varphi,\psi\in \mathcal D(G),
\end{equation*}
where $\mathcal D(M)$ denotes the space of smooth compactly supported functions on a manifold $M$.

Let $E$, $F$ be locally convex spaces (always assumed to be Hausdorff) over $\mathbb{C}$. In general, we will denote the strong dual space of $E$, by $E'$. The dual pairing between $E$ and $E'$, will be denoted by $\langle e',e\rangle:=e'(e)$ for $e'\in E'$, $e\in E$. Similarly we will equip the space of continuous operators from $E$ to $F$ with the topology of uniform convergence on bounded sets of $E$ and denote it by $\mathcal L(E,F)$, resp.\ $\mathcal L(E)$ for $E=F$. For any subspace $A\subset E$ we denote its annihilator by $A^\circ$.
If $E$ and $F$ are Hilbert spaces, the Hilbert-Schmidt operators from $E$ to $F$ will be denoted by $\mathcal{H\! S}(E,F)$ (again $\mathcal{H\! S}(E):=\mathcal{H\! S}(E,E)$). As usual $A^*$ is the adjoint of $A\in\mathcal L(E)$. The trace of some nuclear operator $A\in\mathcal L(E)$, will be denoted by $\operatorname{Tr}[A]$. Furthermore $\mathcal{H\! S}(E,F)$ is equipped with the inner product $(A,B)\mapsto\operatorname{Tr}[AB^*]$. 

Often we need to integrate vector valued functions. For this purpose we will use the concept of weak integrals.
\begin{definition}
	Suppose $(X,\nu)$ is a measure space and $E$ is a locally convex vector space.	We will call a function $f\colon X\to E$ is integrable, iff there is some $e\in E$, such that for each $e'\in E'$ we have $e'\circ f\in L^1(X,\nu)$ and
	\begin{equation*}
		e'(e) = \int_X e'\circ f\,\,\mathrm{d}\nu.
	\end{equation*}
	The element $\int_X f\,\,\mathrm{d}\mu := e$ is called integral over $f$. Usually we will just say, that $\int_X f\,\,\mathrm{d}\mu$ converges in $E$.
\end{definition}

From this definition automatically follows, that
\begin{equation*}
	A\int_X f\,\,\mathrm{d} \nu = \int_X A\circ f\,\,\mathrm{d}\nu.
\end{equation*}
for any continuous linear or antilinear operator $A \colon E\to F$ into another locally convex space $F$. Here, the integrability of $f$ implies the integrability of $A\circ f$.

We denote by $\mathrm{Irr}(G)$ the set of strongly continuous, unitary and irreducible representations of the group $G$. The dual of $G$ is the quotient of $\mathrm{Irr}(G)$ under the equivalence relation of unitary equivalence and is denoted by $\widehat G$. The Plancherel measure to $\mu$, will be $\widehat\mu$. The representation space of some $\pi\in\mathrm{Irr}(G)$ is denoted by $H_\pi$ and the corresponding space of smooth vectors will be denoted by $H_\pi^\infty$. It is equipped with a Fr\'echet topology defined by the seminorms
\begin{equation*}
	v\mapsto \|\pi(P)v\|_{\mathcal H_\pi},\quad\text{for } P\in \mathfrak{u}(\mathfrak{g}_\mathrm{L}),\quad\text{where } \pi(P)v:= P_x\pi(x)v\big|_{x=0}.
\end{equation*}
Finally we will write $H^{-\infty}_\pi$ for the strong dual space of $H^\infty_\pi$, i.e. the dual space of $H^{\infty}_\pi$ equipped with the topology of uniform convergence on bounded sets of $H^{\infty}_\pi$. The group Fourier transform is defined by
\begin{equation*}
	\mathcal F_G\varphi(\pi):=\int_G \varphi(x)\pi(x)^*\,\mathrm{d}\mu(x),\quad\text{for }\varphi\in\mathcal S(G),\ \pi\in\mathrm{Irr}(G),
\end{equation*}
and its inverse reads
\begin{equation*}
	\varphi(x) := \int_{\widehat G}\operatorname{Tr}[\pi(x)\,\mathcal F_G\varphi(\pi)]\,\mathrm{d}\widehat\mu([\pi]),\quad\text{for }\varphi\in\mathcal S(G),\ x\in G.
\end{equation*}
Notice, that for each $\pi\in\mathrm{Irr}(G)$ and each $\varphi\in\mathcal S(G)$ the operator $\mathcal F_G\varphi(\pi)$ is nuclear on $H_\pi$ \cite[Theorem 4.2.1]{rep_nilpotent_lie_groups}. Let $\mathcal S(\widehat G)=\mathcal F_G\mathcal S(G)$ with the final topology induced by $\mathcal F_G$. Since $\mathcal F_G$ is injective, this is the unique topology that makes $\mathcal F_G\colon \mathcal S(G)\to\mathcal S(\widehat G)$ an isomorphism. The space $L^2(\widehat G,\widehat \mu)$ is defined to be the completion of $\mathcal S(\widehat G)$ with respect to the inner product
\begin{equation*}
	\boldsymbol{(} \widehat\varphi,\widehat\psi\boldsymbol{)}_{L^2(\widehat G,\widehat \mu)} := \int_{\widehat G}\operatorname{Tr}[\widehat\varphi(\pi) \,\widehat\psi(\pi)^*]\,\mathrm{d}\widehat\mu([\pi]),\quad\text{for }\widehat\varphi,\widehat\psi\in\mathcal S(\widehat G).
\end{equation*}
The Fourier transform extends to a unitary operator between $L^2(G,\mu)$ and $L^2(\widehat G,\widehat \mu)$. See e.g. \cite{ruzhansky_nilpotent} for more details.

\subsection{Tensor products}

By $E\otimes F$ we denote the algebraic tensor product between $E$ and $F$. Their complete injective tensor product will be denoted by $E\,\hat\otimes_\varepsilon F$ and their complete projective tensor product by $E\,\hat\otimes_\pi F$. If $E_j$ and $F_j$, $j=1,2$, are locally convex spaces and $A_j\in\mathcal L(E_j,F_j)$, then
\begin{equation*}
	A_1\otimes A_2 \colon E_1\,\hat\otimes_\varepsilon E_2 \to  \colon F_1\,\hat\otimes_\varepsilon F_2
\end{equation*}
denotes the tensor product map of $A_1$ and $A_2$. The linear map $A_1\otimes A_2$ is continuous and is even an isomorphism, if $A_1$ and $A_2$ are isomorphisms \cite[Proposition 43.7]{treves_tvs}. Notice, $A_1\otimes A_2$ can also be defined, if $A_1$ and $A_2$ are continuous anti-linear operators. Later, we will need the following Lemma.

\begin{lemma}\label{lemma:allg_tenserop_stetig_alt}
	Let $E,F$ and $G$ be locally convex spaces, then
	\begin{equation*}
		\mathcal L(E,G) \to \mathcal L(E\,\hat\otimes_\varepsilon\, F, G\,\hat\otimes_\varepsilon\,F) \colon A\mapsto A\otimes 1
	\end{equation*}
	is continuous.
\end{lemma}
\begin{proof}
	The topology on $\mathcal L(E\,\hat\otimes_\varepsilon\, F, G\,\hat\otimes_\varepsilon\,F)$ is induced by seminorms of the form
	\begin{equation*}
		A \mapsto \sup_{z\in B} p(Az)
	\end{equation*}
	where $B$ is a bounded set in $E\,\hat\otimes_\varepsilon\, F$ and $p$ is a continuous seminorm on  $G\,\hat\otimes_\varepsilon\,F$. For $p$ it is sufficient to take any semi norm of the form
	\begin{equation*}
		p(z):= \sup_{\phi\in V}q((1\otimes \phi)(z))
	\end{equation*}
	where $V$ is an equicontinuous subset of $F'$ and $q$ a continuous seminorm on $G$ \cite[Definition 43.1 and Proposition 36.1]{treves_tvs}. Notice that the set $B\subset E\,\hat\otimes_\varepsilon\, F$ is bounded, iff for all equicontinuous sets $W\subset F'$ and all continuous seminorms $r$ on $E$
	\begin{equation*}
		\sup_{\phi\in W}\sup_{z\in B} r (1\otimes \phi (z))<\infty.
	\end{equation*}
	In general, a subset of $E$ is bounded, iff all continuous seminorms $r$ are bounded on the set. Hence the set $B_V := \cup_{\phi\in V}(1\otimes \phi)(B)$ is a bounded subset of $E$. We arrive at
	\begin{equation*}
		\sup_{z\in B} p((A\otimes 1)z) = \sup_{z\in B} \sup_{\phi\in V}q((A\otimes \phi)(z)) = \sup_{\phi \in V}\sup_{e\in (1\otimes \phi)B}q(Ae) =  \sup_{e\in B_V} q(Ae),
	\end{equation*}
	where the right hand side defines a continuous seminorm on $\mathcal L(E,G)$.
\end{proof}

If either $E$ or $F$ are nuclear, both tensor product topologies result in the same locally convex space and we may just write $E\,\hat\otimes\, F$ for either of the complete tensor products of $E$ and $F$ \cite[Theorem 50.1]{treves_tvs}. If both $E$ and $F$ are nuclear Fr\'echet spaces, then so is $E\,\hat\otimes\, F$ \cite[Proposition 50.1]{treves_tvs} and \cite[Chapter III corollary to 6.3]{schaefer_tvs}. One reason for the usage of nuclear spaces is the following abstract kernel theorem \cite[Propositions 50.5 - 50.7]{treves_tvs}.

\begin{theorem}\label{lemma:allg_kernsatz}
	If $F$ is a complete locally convex space and $E$ is a nuclear Fr\'echet space or dual to a nuclear Fr\'echet space, then $F\,\hat\otimes\, E'\cong \mathcal L(E,F)$ via the extension of the canonical map
	\begin{equation*}
		\sum_{j} f_j\otimes e'_j \mapsto \big[e\mapsto \sum_{j} f_j\,e'_j(e)\big].
	\end{equation*}
	Suppose additionally both $E$ and $F$ are Fr\'echet spaces, then $F'\,\hat\otimes\, E'\simeq (F\,\hat\otimes\, E)' $, via the extension of the map
	\begin{equation*}
		\sum_{k}f'_k\otimes e'_k\mapsto \big[\sum_{j} f_j\otimes e_j \mapsto \sum_{j,k} f'_k(f_j)\,e'_k(e_j)\big].
	\end{equation*}
\end{theorem}

The multiplication between spaces of smooth functions and spaces of distributions is rarely a continuous bilinear map. But often it is hypocontinuous. Here, a bilinear map $u\colon E\times F\to G$ is defined to be hypocontinuous between the locally convex spaces $E,F$ and $G$, if for all bounded sets $B_E\subset E$ and $B_F\subset F$, the two sets of linear maps
\begin{equation*}
	\{u(e,\cdot)\mid e\in B_E\}\quad \text{and}\quad \{ u(\cdot,f)\mid f\in B_F\}
\end{equation*}
are equicontinuous.

Linear maps on tensor factors can easily be combined to construct a linear map on the complete tensor product. The situation for bilinear maps is not as simple. However, in the context of nuclear spaces, we may use the following theorem, which is an amalgamation of a proposition of C. Bargetz and N. Ortner and a corollary of L. Schwartz.

\begin{theorem}\label{theorem:allg_fortsetzung_bilinear}
	Let $\mathcal H$, $\mathcal K$ and $\mathcal L$ be complete nuclear locally convex spaces with nuclear strong duals and let $E$, $F$ and $G$ be complete locally convex spaces. Suppose that
	\begin{equation*}
		u\colon \mathcal H\times\mathcal K\to\mathcal L\quad\text{and}\quad b\colon E\times F\to G
	\end{equation*}
	are two hypocontinuous bilinear maps. Suppose furthermore, that either one of the three properties
	\begin{itemize}
		\item $\mathcal H$ and $E$ are Fr\'echet spaces
		\item $\mathcal H$ and $E$ are strong duals of Fr\'echet spaces
		\item the bilinear map $b$ is continuous
	\end{itemize}
	are fulfilled. Then there is a unique hypocontinuous bilinear map
	\begin{equation*}
		{}^b_u \colon (\mathcal H\,\hat\otimes\, E)\times (\mathcal K\,\hat\otimes\, F) \to \mathcal L\,\hat\otimes\, G,
	\end{equation*}
	that fulfils the consistency property
	\begin{equation*}
		{}^b_u(S\otimes e,T\otimes f) = u(S,T)\otimes b(e,f).
	\end{equation*}
\end{theorem}
\begin{proof}

For the cases, where $\mathcal H$ and $E$ are both Fr\'echet or both duals to Fr\'echet spaces, this statement can be found in \cite[Proposition 1]{bargetz_char_of_mult_spaces}. For the case, where $b$ is continuous, we find the statement in \cite[Corollair and Remarques on page 38]{Schwartz_distr_vector_II}. However, in the sources the notation $\mathcal H(E):=\mathcal H\varepsilon E$ for the $\varepsilon$-product of Schwartz is used. The nuclearity and completeness of $\mathcal H$ and the completeness of $E$ make sure, that $\mathcal H\varepsilon E = \mathcal H \,\hat\otimes\, E$ by \cite[Satz 10.17 and Satz 11.18]{aufbaukurs}, which fits our notation.
\end{proof}

Examples of spaces fulfilling the conditions for $\mathcal H$, $\mathcal K$ and $\mathcal L$, are $\mathcal S(\mathbb{R}^n)$ $\mathcal S'(\mathbb{R}^n)$, $\mathcal O_\mathrm{M}(\mathbb{R}^n)$ \cite[Chapitre II Th\'eor\`eme 16]{grothendieck_tensorprodukte}, $\mathcal L(\mathcal S(\mathbb{R}^n))\simeq \mathcal L(\mathcal S'(\mathbb{R}^n))$ and $\mathcal O_\mathrm{M}(\mathbb{R}^n)\,\hat\otimes\,\mathcal L(\mathcal S(\mathbb{R}^m))$ \cite{BARGETZ_bases}.

The Hilbert space tensor product of $E$ with $F$ will be denoted by $E\,\hat\otimes_2\, F$. By a slight abuse of notation, we will also denote
\begin{equation*}
	A_1\otimes A_2 \colon E_1\,\hat\otimes_2\, F_2\to F_1\,\hat\otimes_2\, E_2
\end{equation*}
to be the tensor product map of continuous linear maps $A_j$ between the Hilbert spaces $E_j$ and $F_j$. If $A_1$ and $A_2$ are unitary, then so is $A_1\otimes A_2$.

\subsection{Gelfand triples}

Gelfand triples are a convenient setting for both distributions and the Fourier transform. We start by defining the class of Gelfand triples we are going to use.

\begin{definition}
	A Gelfand triple (with a real structure) is a tuple of spaces $\mathcal G = (E,G,E')$, and a structure map $\mathcal C\colon H\to H$ fulfilling the following properties:
	\begin{enumerate}
		\item[\bf(a)] $E$ is a nuclear Fr\'echet space, with strong dual $E'$
		\item[\bf(b)] $H$ is a Hilbert space, with dense and continuous embedding $E\hookrightarrow H$
		\item[\bf(c)] $\mathcal C$ is antiunitary, $\mathcal C^2=1$ and $\mathcal C{|_E}\colon E\to E$ is a homeomorphism.
	\end{enumerate}
	The map $\mathcal C$ will be called the real structure of $\mathcal G$.
\end{definition}
Notice, that by the definition of Gelfand triples we automatically have a continuous dense embedding $H'\hookrightarrow E'$ dual to the embedding $E\hookrightarrow H$. Classically Gelfand triples are defined without the structure map $\mathcal C$ \cite{gelfand_genfct_4}. Here $E$ and $H$ are antilinearly embedded into $E'$ via the Fr\'echet-Riesz isomorphism $\mathcal R\colon H\xrightarrow{\simeq} H'$. But this approach would be unwieldy, because we will use Gelfand triples in concert with tensor products. Since there is no canonical unitary map between $H$ and $H'$, we are going to use $\mathcal C$ to fix one. The corresponding induced embedding $\mathcal I\colon H\hookrightarrow E'$ is defined via
\begin{equation*}
	\mathcal I \colon H\xrightarrow{\mathcal R\circ \mathcal C} H'\hookrightarrow E'.
\end{equation*}
The structure map $\mathcal C$ has a natural continuation to a homeomorphism $\mathcal C\colon E'\to E'$ by $(\mathcal Ce')(e):=\overline{e'(\mathcal C e)}$ for $e\in E$, $e'\in E'$. Hence $\mathcal C$ induces compatible real structures on $E$, $H$ and $E'$.

In the sequel we will always use the term Gelfand triple for the concept of a Gelfand triple with real structure equipped with embeddings as described above.

\begin{definition}
	Let $\mathcal G_j = (E_j,H_j,E'_j)$, $j=1,2$, be Gelfand triples. For a map $T\colon E'\to E'$, we write
	\begin{equation*}
		T\colon \mathcal G_1\to \mathcal G_2,
	\end{equation*}
	if $T(E_1)\subset E_2$ and $T(H_1)\subset H_2$ with respect to the above described embeddings. We will call $T$ a Gelfand triple isomorphism, if $T|_{E_1}\colon E_1\to E_2$, $T\colon E_1'\to E_2'$ are homeomorphisms and $T_{H_1}\colon H_1\to H_2$ is unitary.
\end{definition}
The above definition implies, that writing $T\colon \mathcal G_1\to \mathcal G_2$ is equivalent to saying, that the diagram
\begin{center}
\begin{tikzcd}
	E_1 \arrow[r,hook] \arrow[d,"T"'] 
		& H_1 \arrow[r,hook] \arrow[d,"T"']
		& E'_1 \arrow[d,"T"'] \\
	E_2 \arrow[r,hook]
		& H_2 \arrow[r,hook]
		& E'_2
\end{tikzcd}
\end{center}
is commutative.

Now we will describe, how we may construct new Gelfand triples.
\begin{definition}
	Let $\mathcal G_j = (E_j,H_j,E'_j)$ be Gelfand triples with structure maps $\mathcal C_j$ for $j=1,2,3,4$.
	Using the identifications $E_1'\oplus E_2'\simeq (E_1\oplus E_2)'$ and $E_1'\,\hat\otimes\, E_2'\simeq (E_1\oplus E_2)'$ resp.\ $\mathcal L(E_2,E_1') \simeq \mathcal L(E_2',E_1)'$ via Theorem \ref{lemma:allg_kernsatz} we may define the following Gelfand triples.
	
	The sum resp.\ tensor product of $\mathcal G_1$ and $\mathcal G_2$ is defined by 
	\begin{equation*}
		\mathcal G_1\oplus \mathcal G_2:=\left(\begin{matrix}
		E_1\oplus E_2\\
		H_1 \oplus H_2 \\
		E_1'\oplus E_2'
		\end{matrix}\right)\quad\text{resp.}\quad \mathcal G_1\otimes\mathcal G_2:=\left(\begin{matrix}
		E_1\,\hat\otimes\, E_2\\
		H_1\,\hat\otimes_2\,H_2 \\
		E_1'\,\hat\otimes\,E_2'
		\end{matrix}\right)
	\end{equation*}
	with structure maps $\mathcal C_1\oplus \mathcal C_2$ resp.\ $\mathcal C_1\otimes\mathcal C_2$. Here $\mathcal C_1\otimes \mathcal C_2$ is the continuation of
	\begin{equation*}
		H_1\otimes H_2 \to H_1\otimes H_2\colon \sum_{j,k}  h_{1,j}\otimes  h_{2,k}\mapsto \sum_{j,k} (\mathcal C_1 h_{1,j})\otimes (\mathcal C_2 h_{2,k}).
	\end{equation*}
	The operator Gelfand triple from $\mathcal G_2$ to $\mathcal G_1$ is defined as
	\begin{equation*}
		\mathcal L(\mathcal G_2,\mathcal G_1):=\left(\begin{matrix}
		\mathcal L (E_2', E_2)\\
		\mathcal {H\!S}(H_2,H_1') \\
		\mathcal L(E_2,E_1')
		\end{matrix}\right)
	\end{equation*}
	with structure map $T\mapsto \mathcal C_1\, T\,\mathcal C_2$.
\end{definition}

Let us now discuss, why $\mathcal G_1\otimes \mathcal G_2$ and $\mathcal L(\mathcal G_2,\mathcal G_1)$ are indeed Gelfand triples. Notice, that by using the unitary isomorphism $\mathcal I_2 \colon H_2\xrightarrow{\simeq} H_2'$ determined by $\mathcal C_2$, we get an isomorphism resp.\ a unitary isomorphism
\begin{equation*}
	\mathcal{N}(H_2,H_1)\simeq H_1\,\hat\otimes_\pi\, H_2 \quad\text{resp.}\quad \mathcal{H\!S}(H_2,H_1) \simeq H_1\,\hat\otimes_2\, H_2,
\end{equation*}
by extending the linear map
\begin{equation*}
	\mathcal K^{-1} \colon H_1\otimes H_2 \to \mathcal{N}(H_2,H_1),\quad\text{where} \quad\mathcal{K}^{-1}(h_1\otimes h_2)(\tilde h_2):= h_1\cdot (\mathcal I_2 h_2,)(\tilde h_2)
\end{equation*}
and $\mathcal N(H_2,H_1)$ is the space of nuclear operator from $H_2$ to $H_1$. These isomorphisms, together with Theorem \ref{lemma:allg_kernsatz}, can be used to construct the following chain of continuous maps with dense ranges
\begin{center}
\begin{tikzcd}
	E_1\,\hat\otimes\, E_2  \arrow[r,hook]
		& H_1\,\hat\otimes_\pi\,H_2 \arrow[d,"\mathcal K^{-1}"',"\simeq"]
		& H_1\,\hat\otimes_2\, H_2\\
	\mathcal L(E_2',E_1) \arrow[u,"\mathcal K", "\simeq"']
		& \mathcal N(H_2,H_1) \arrow[r,hook]
		& \mathcal{H\!S}(H_2,H_1) \arrow[u,"\mathcal K","\simeq"']
\end{tikzcd}
\end{center}
We get a commutative diagram by completing the rows of this diagram. This both shows, that $\mathcal L(\mathcal G_2,\mathcal G_1)$ and $\mathcal G_1\otimes\mathcal G_2$ are Gelfand triples, and proves the following Lemma.

\begin{lemma}\label{lemma:allg_kernsatz_gelfand_triple}
	Suppose $\mathcal G_1$ and $\mathcal G_2$ are the Gelfand triples from above, then the canonical linear map
	\begin{equation*}
		\mathcal K^{-1}\colon E_1\otimes E_2\to \mathcal L(E_2',E_1),\quad\text{where}\quad \mathcal K^{-1}(e_1\otimes e_2)(e'_2):=e_1\cdot e_2'(e_2),
	\end{equation*}
	induces a Gelfand triple isomorphism $\mathcal K \colon \mathcal L(\mathcal G_2,\mathcal G_1)\to\mathcal G_1\otimes \mathcal G_2$.
\end{lemma}

Since we choose $G$ simply connected and nilpotent, the space $H_\pi^\infty$ is a nuclear Fr\'echet space \cite[Corollary 4.1.2]{rep_nilpotent_lie_groups} for any $\pi\in\mathrm{Irr}(G)$. With respect to $\pi$ we may define the Gelfand triples
\begin{equation*}
	\mathcal G(\pi):=\left(\begin{matrix} H_\pi^\infty\\ H_\pi \\ H_\pi^{-\infty} \end{matrix}\right)\quad\text{and}\quad \mathcal G_\mathrm{op}(\pi):= \mathcal L(\mathcal G(\pi),\mathcal G(\pi)) =\left(\begin{matrix} \mathcal L(H_\pi^{-\infty}, H_\pi^\infty)\\ \mathcal{H\!S}( H_\pi) \\ \mathcal L( H_\pi^{\infty}, H_\pi^{-\infty}) \end{matrix}\right),
\end{equation*}
if we associate a real structure $\mathcal C_\pi$ to $\mathcal G(\pi)$. Of course, this real structure is not unique. Instead, we will define $\mathrm{Irr}^\mathbb{R}(G)$ to be pairs consisting of $\pi\in\mathrm{Irr}(G)$ and an associated real structure $\mathcal C_\pi$ on $\mathcal G(\pi)$. Usually we will just write $\pi\in\mathrm{Irr}^\mathbb{R}(G)$ and mean, that we took a choice of $\mathcal C_\pi$ for $\pi$.

\section{Polynomial manifolds and Gelfand triples for the Fourier transform}
\label{subsection:pol_and_mod_structures}

We will need polynomial manifolds for the Pedersen quantization, a generalization of the Weyl quantization, and for the generalizations of the spaces $\mathcal S(\mathbb{R}^n;E)$ and $\mathcal O_\mathrm{M}(\mathbb{R}^n;E)$ for a complete locally convex space $E$, see \cite{Schwartz_distributions,grothendieck_tensorprodukte}. For it will be convenient to have one notion of Schwartz functions and slowly increasing functions, that can be applied to simply connected nilpotent Lie groups, Lie algebras, coadjoint orbits, and that is also compatible with their relation to each other. Furthermore this will lead to a notion of Schwartz functions and slowly increasing function on $\mathbb R^\times = \mathbb R\setminus\{0\}$, that we will rely on heavily.  

Before we define polynomial manifolds, let us fix some basic notation and recall a few definitions from the theory of vector valued smooth functions.

As usual we will just write $\mathcal S(\mathbb{R}^n) :=\mathcal S(\mathbb{R}^n;E)$ and $\mathcal O_\mathrm{M}(\mathbb{R}^n):=\mathcal O_\mathrm{M}(\mathbb{R}^n;E)$ for scalar valued case $E=\mathbb{C}$.

A function $f\colon \mathbb R^n\to E$ is differentiable, if for each $j=1,\dotsc,n$ and each $x\in\mathbb R^n$ the limit
\begin{equation*}
	\partial_jf(x) := \lim_{t\to 0}\frac{1}{t}(f(x+t e_j)-f(x))
\end{equation*}
exists in $E$, where $(e_j)_j$ is the standard basis in $\mathbb R^n$, and each partial derivative, $\partial_j f$, is a continuous function. $f$ is called smooth, if partial derivatives of arbitrary order are again differentiable. I.e. for all $\alpha\in\mathbb N_0^n$ the functions $\partial^\alpha f = \partial^{\alpha_1}_1\cdots \partial^{\alpha_n}_n f$ exist and are continuous.

Denote by $\mathcal P(\mathbb{R}^n)$ the vector space of polynomial functions from $\mathbb{R}^n$ to $\mathbb{C}$ and by $\mathrm{Diff}_\mathcal{P}(\mathbb{R}^n)$ the set of Differential operators with polynomial coefficients on $\mathbb{R}^n$. The space of $E$-valued Schwartz functions, $\mathcal S(\mathbb{R}^n;E)$, is the space of smooth functions $\varphi\colon \mathbb{R}^n\to E$, such that
\begin{equation*}
	\sup_{x\in\mathbb{R}^n} p(P\varphi(x))<\infty
\end{equation*}
for each continuous seminorm $p$ on $E$ and each $P\in\mathrm{Diff}_\mathcal{P}(\mathbb{R}^n)$. The above expression also defines a set of seminorms which define the topology on $\mathcal S(\mathbb{R}^n;E)$.

The scalar valued slowly increasing functions, $\mathcal O_\mathrm{M}(\mathbb R^n)$, is the space of smooth functions $f\colon \mathbb R^n\to C$, such that for each $P\in\mathrm{Diff}\mathcal{P}(\mathbb{R}^n)$ there is a real valued $q\in \mathcal P(\mathbb R^n)$, such that $|P f|\le q$.

The space of $E$-valued slowly increasing functions, $\mathcal O_\mathrm{M}(\mathbb{R}^n;E)$, is the space of smooth functions $f\colon \mathbb{R}^n\to E$, such that
\begin{equation*}
	[\varphi\mapsto f\cdot \varphi]\in\mathcal L(\mathcal S(\mathbb{R}^n),\mathcal S(\mathbb{R}^n;E)),
\end{equation*}
equipped with the subspace topology in $\mathcal L(\mathcal S(\mathbb{R}^n),\mathcal S(\mathbb{R}^n;E))$. For $E=\mathbb C$ the two given definitions are equivalent \cite[Theorem 25.5]{treves_tvs}.

Furthermore the above defined spaces of $E$-valued functions have a very useful characterization by tensor products. To be precise, we have
 $\mathcal S(\mathbb{R}^n;E) = \mathcal S(\mathbb{R}^n)\,\hat\otimes\, E$ and $\mathcal O_\mathrm{M}(\mathbb{R}^n;E) = \mathcal O_\mathrm{M}(\mathbb{R}^n)\,\hat\otimes\, E$ as topological vector spaces.

The definition given below is a slight generalization of the polynomial manifolds used by Pedersen in \cite{pedersen_geometric_quant}.
\begin{definition}
Suppose now $M$ is an $n$-dimensional smooth manifold with finitely many connected components. An atlas $\mathcal A$ of $M$ will be called a polynomial atlas, iff each two charts $(\phi,U),(\psi,V)\in\mathcal A$ fulfil
\begin{enumerate}
	\item[\bf(i)] $U$, $V$ are connected components of $M$ and $\phi(U)=\psi(V)=\mathbb{R}^n$,
	\item[\bf(ii)] and if $U=V$, then $\phi\circ\psi^{-1}$ is a polynomial function on $\mathbb{R}^n$.
\end{enumerate}
Two polynomial atlases $\mathcal A,\mathcal A'$ are said to be equivalent, iff $\mathcal A\cup\mathcal A'$ is a polynomial atlas. A polynomial structure is an equivalence class of polynomial atlases.

Together with a polynomial structure $M$ will be called a polynomial manifold. A chart of a polynomial structure of $M$ will be called a polynomial chart on $M$.
\end{definition}

The following list provides some basic examples of polynomial manifolds:
\begin{itemize}
	\item a finite dimensional vector space, with respect to the linear charts
	\item a finite dimensional affine spaces, with respect to the affine linear charts
	\item a simply connected nilpotent Lie group, with respect to the exponential map
	\item a coadjoint orbit to a simply connected nilpotent Lie group \cite{pedersen_geometric_quant}
\end{itemize}

By using polynomial charts, we may generalize polynomials and definitions that depend on the set of polynomials.

\begin{definition}
	Suppose $M,N$ are polynomial manifolds and $E$ a complete locally convex space. For $\mathcal X(\bullet)\in \{\mathcal P(\bullet),\mathcal S(\bullet ; E), \mathcal O_\mathrm{M}(\bullet ; E)\}$ we define
	\begin{equation*}
		\mathcal X(M):= \{f\colon M\to E\mid f\circ \phi^{-1}\in\mathcal X(\mathbb R^n)\text{ for all polynomial charts }\phi\}
	\end{equation*}
	
	We equip $\mathcal S(M;E)$ resp.\ $\mathcal O_\mathrm{M}(M;E)$ with the projective topology from the maps $f\mapsto f\circ\phi^{-1}$ into $\mathcal S(\mathbb R^n;E)$ resp.\ $\mathcal O_\mathrm{M}(\mathbb R^n;E)$. As usual we set $\mathcal S(M):=\mathcal S(M;\mathbb C)$ and $\mathcal O_\mathrm{M}(M):=\mathcal O_\mathrm{M}(M;\mathbb C)$.
	
	The set of polynomial differential operators on $M$ is defined to be
	\begin{equation*}
		\mathrm{Diff}_\mathcal{P}(M):= \left\lbrace P\in\mathcal L(\mathcal D(M)) : \begin{matrix}\left[ \varphi\mapsto P(\varphi\circ \phi)\circ \phi^{-1}\right]\in\mathrm{Diff}_\mathcal{P}(\mathbb{R}^n)\\ \text{for all polynomial charts }\phi \end{matrix} \right\rbrace.
	\end{equation*}
	A function $f\colon M\to N$ will be called polynomial resp.\ slowly increasing, iff $\psi\circ f\circ \phi^{-1}$ is a polynomial resp.\ slowly increasing $(\phi,U)$ on $M$ and $(\psi,V)$ on $N$ with $U\subset f^{-1}(V)$. The function $f$ will be called polynomial resp.\ tempered diffeomorphism, iff $f$ is bijective and both $f$ and $f^{-1}$ are polynomial resp.\ slowly increasing.
\end{definition}

As for smooth manifolds, we may construct new polynomial manifolds by disjoint unions $M\dot\cup N$ of polynomial manifolds $N$, $M$ with the same dimension and products $M\times N$ of arbitrary polynomial manifolds $N$, $M$. The corresponding polynomial structure on $M\dot\cup N$ is induced by the polynomial charts on $M$ and $N$. On $M\times N$ we choose the canonical polynomial structure defined by combining charts $\phi$ on $M$ and $\psi$ on $N$ to polynomial charts $(\phi,\psi)$ on $M\times N$. Directly from our definition follows, that similar to the euclidean case
\begin{equation*}
	\mathcal S(M\dot\cup N) = \mathcal S (M)\oplus \mathcal S(N)\quad\text{and}\quad \mathcal S(M\times N; E) = \mathcal S(M)\,\hat\otimes\,\mathcal S(N)\,\hat\otimes\, E,
\end{equation*}
where $E$ is a complete locally convex space. The identities also hold, if we exchange $\mathcal S$ with $O_\mathrm{M}$. Similar identities are true for $\mathcal P(M)$ and $\mathrm{Diff}_\mathcal{P}(M)$.

We will call a Radon measure $\nu$ on $\mathbb{R}^n$ tempered, iff it is equivalent to the Lebesgue measure $\,\mathrm{d} x$ and the Radon-Nikodym derivatives $\frac{\,\mathrm{d} x}{\,\mathrm{d} \nu}$ and $\frac{\,\mathrm{d} \nu}{\,\mathrm{d} x}$ are slowly increasing almost everywhere. A Radon measure on a polynomial manifold $\mathbb{R}^n$ will be called tempered, if each pushforward by a polynomial chart is tempered.

\begin{definition}\label{definition:gelfand_triple_pol_mfkt}
	Suppose $M$ is a polynomial manifold and $\nu$ a tempered measure on $M$. Then $\mathcal G(M,\nu)$ is defined to be the Gelfand triple 
\begin{equation*}
	\mathcal S(M)\hookrightarrow L^2(M,\nu)\hookrightarrow \mathcal S'(M),
\end{equation*}
	equipped with the real structure defined by the usual complex conjugation $\varphi\mapsto \overline{\varphi}$.
\end{definition}
If $f\colon M_1\to M$ is a tempered diffeomorphism, then for each $\phi\in\mathcal S'(M)$ the pull back $\wp_f\phi(\varphi):=\phi(\varphi\circ f^{-1})$ is well defined and induces a Gelfand triple isomorphism
\begin{equation*}
	\mathcal G(M,\nu) \to \mathcal G(M_1,\nu\circ f^{-1}).
\end{equation*}

Indeed, we defined tempered measures and polynomial manifolds in such a way, that we have a very simple Gelfand-Triple isomorphism
\begin{equation*}
	\mathcal{G}(M,\nu) \simeq \bigoplus_{j=1}^k \mathcal G(\mathbb{R}^n,\,\mathrm{d} x),
\end{equation*}
given by pullbacks and multiplications with slowly increasing functions, provided that $M$ is a $n$-dimensional polynomial manifold with $k$ connected components.

\subsection{The polynomial manifold $\mathbb{R}^\times$}

For us the two most important examples of polynomial manifolds are the half lines $\mathbb{R}^+ $ and $\mathbb{R}^-$. Here the polynomial structure is induced by the chart $\sigma \colon \lambda\mapsto |\lambda|-1/|\lambda|$. On $\mathbb{R}^+$ the inverse reads $\sigma^{-1}(y) = (y+ \sqrt{y^2 + 4})/2$.

\begin{lemma}\label{lemma:reducedSchwartzspace_carries_SubTop}
	If we extend each function in $\mathcal S(\mathbb{R}^\pm)$ by zero to the whole real line, then
	\begin{equation*}
		\mathcal S(\mathbb{R}^\pm)= \{\varphi\in \mathcal S(\mathbb{R})\mid \varphi\equiv 0 \text{ on }\mathbb{R}^{\mp}\}
	\end{equation*}
	and $\mathcal S(\mathbb{R}^\pm)$ carries the subspace topology with respect to $\mathcal S(\mathbb{R})$.
\end{lemma}
\begin{proof}
	We will prove the statement for the $\mathbb{R}^+$ case, for $\mathbb{R}^-$ the proof is analogous. 
	 Since $\sigma$ is a polynomial diffeomorphism from $\mathbb R^+$ to $\mathbb R$, the map
	\begin{equation*}
		\varphi \mapsto \varphi \circ \sigma
	\end{equation*}
	is a linear homeomorphism between $\mathcal D(\mathbb R)$ and $\mathcal D(\mathbb R^+)$ resp. between $\mathcal S(\mathbb R)$ and $\mathcal S(\mathbb R^+)$. Hence $\mathcal D(\mathbb R^+)$ is dense in $\mathcal S(\mathbb R^+)$. Let us define
	\begin{equation*}
		\mathcal S_+(\mathbb{R}):= \{\varphi\restriction_{\mathbb R^+}\ \mid\ \varphi\in \mathcal S(\mathbb{R})\text{ with }\varphi\equiv 0 \text{ on }\mathbb{R}^{-}\},
	\end{equation*}
	equipped with the subspace topology with respect to $\mathcal S(\mathbb{R})$. Let $f\colon \mathbb R\to [0,1]$ be smooth, such that $\operatorname{supp} f\subset \mathbb R^+$ and $f\equiv 1$ on $[1,\infty)$. For each $\varphi \in \mathcal S_+(\mathbb R)$ and $\alpha\in\mathbb N_0$ we have $\partial^\alpha\varphi(x) = o(x^N)$ for $x\to 0$ of arbitrary high order $N\in\mathbb N$. Hence each $\alpha,\beta\in\mathbb N_0$ there is some $C_1,C_2>0$ and $N>\alpha$ with
	\begin{align*}
		\sup_{x\in\mathbb R^+}|x^\beta \partial_x^\alpha (f(n\,x)\,\varphi(x) - \varphi(x))| &\le C_1\sum_{0<\gamma\le \alpha} n^\gamma\sup_{y\in\mathbb R}|\partial^\gamma f(y)| \sup_{0<x\le 1/n} x^\beta |\partial^{\alpha-\gamma}\varphi(x)|\\
		&\le C_2 \sum_{0<\gamma\le \alpha} n^{\gamma-\beta-N}\xrightarrow{n\to\infty} 0.
	\end{align*}
	By employing the usual cut-off functions, we realize that $\mathcal D(\mathbb R^+)$ is dense in $\mathcal S_+(\mathbb R)$, too. Thus it is enough to show that the topologies of $\mathcal S(\mathbb{R}^+)$ and $\mathcal S_+(\mathbb{R})$ coincide on $\mathcal D(\mathbb{R}^+)$. The $\mathcal S(\mathbb{R}^+)$-topology is induced by seminorms of the form
	\begin{equation*}
		\mathcal D(\mathbb{R}^+)\to\mathbb{R}\colon \varphi \mapsto \sup_{x>0}|A^kB^j \varphi(x)|,\quad k,j\in\mathbb{N}_0,
	\end{equation*}
	where $A\varphi := (\partial(\varphi\circ\sigma^{-1}))\circ\sigma$ and $B\varphi := (\boldsymbol m(\varphi\circ\sigma^{-1}))\circ\sigma = \sigma\cdot \varphi$. First of all, we have
	\begin{equation*}
		A\varphi(x) = \frac{x^2}{x^2 + 1}\partial_x\varphi(x)=:\eta(x)\cdot \partial_x\varphi(x),\quad \varphi\in\mathcal D(\mathbb{R}_+),x\in\mathbb{R}_+.
	\end{equation*}
	The rational function $\eta$ and all of its derivatives are bounded. Hence $A$ can be extended to an operator in $\mathcal L(\mathcal S_+(\mathbb{R}))$. 
	
	We show, that $B$ has an extension in $\mathcal L(\mathcal S_+(\mathbb{R}))$. For this purpose it is enough to prove that $\frac{1}{\boldsymbol m}\in \mathcal L(\mathcal S_+(\mathbb{R}))$, where $\frac{1}{\boldsymbol m}\varphi(x) = \varphi(x)/x$. First of all, for each $\varphi\in\mathcal D(\mathbb{R}^+)$ and each $x>1$
	\begin{equation*}
	\begin{split}
		|x^k\partial^n_x(\frac{1}{x}\varphi(x))| &\le \sum_{j=0}^n\frac{n!}{(n-j)!} x^{k-j-1}|\varphi^{(n-j)}(x)| \\
		&\le \sum_{j=0}^n\frac{n!}{(n-j)!}\sup_y|y^{k}\varphi^{(n-j)}(y)|,
	\end{split}
	\end{equation*}
	for arbitrary $k,n\in \mathbb{N}_0$. Now we only need to bound the left-hand side for $0<x<1$. For $k>n$ almost the same inequality as above can be used. We assume now $n\ge k$. For $0<x<1$ and each $m\in\mathbb{N}$
	\begin{equation*}
		|\varphi(x)/x^m| = |\frac{1}{x^m} \int_0^x \frac{(x-t)^{m-1}}{(m-1)!}\varphi^{(m)}(t)\,\,\mathrm{d} t | \le \frac{1}{m!}\,\sup_y|\varphi^{(m)}(y)|.
	\end{equation*}
	Hence
	\begin{equation}
	\begin{split}
		|x^k\partial^n_x(\frac{1}{x}\varphi(x))| &\le \sum_{j=0}^n\frac{n!}{(n-j)!} x^{k-j-1}|\varphi^{(n-j)}(x)|\\
		&\le \sum_{j=0}^n\frac{n!}{(n-j)!}x^{-n-1}|\varphi^{(n-j)}(x)|\label{eq:proof1overx_cont}\\
		&\le \sum_{j=0}^n\frac{1}{(n-j)!(n+1)}\sup_y|\varphi^{(2n+1-j)}(y)|,
	\end{split}
	\end{equation}
	for all $0<x<1$, $n\le k$ and $\varphi\in\mathcal D(\mathbb{R}^+)$. In conclusion $\frac{1}{\boldsymbol m}\in \mathcal L(\mathcal S_+(\mathbb{R}))$ and thus also $B\in \mathcal L(\mathcal S_+(\mathbb{R}))$. Due to the continuity of $A$ and $B$ we arrive at
	\begin{equation*}
		\mathcal S_+(\mathbb{R}) \hookrightarrow \mathcal S(\mathbb{R}^+),
	\end{equation*}
	i.e. the $\mathcal S_+(\mathbb{R})$-topology is finer than the $\mathcal S(\mathbb{R}^+)$-topology.
	
	For the reverse embedding we will transport our situation to the whole real line by
	\begin{equation*}
		\varphi \mapsto \varphi \circ \sigma^{-1},
	\end{equation*}
	which is an isomorphism $\mathcal D(\mathbb{R}^+)\to \mathcal D(\mathbb{R})$ and $\mathcal S(\mathbb{R}^+) \to \mathcal S(\mathbb{R})$. We denote the image of $\mathcal S_+(\mathbb{R})$ by this map by $\mathcal S_\oplus (\mathbb{R})$ and equip it with the transported $\mathcal S_+(\mathbb{R})$-topology. Then $\mathcal S_\oplus(\mathbb R)$ is a space of smooth functions on $\mathbb R$ with
	\begin{equation*}
		\mathcal D(\mathbb{R})\hookrightarrow \mathcal S_\oplus (\mathbb{R})\hookrightarrow \mathcal S (\mathbb{R}),
	\end{equation*}
	where both embeddings are dense. The topology in $\mathcal S_\oplus (\mathbb{R})$ is induced by seminorms of the form
	\begin{equation*}
		\mathcal S_\oplus(\mathbb{R}) \to \mathbb{R}\colon \varphi\mapsto \sup_{y\in\mathbb{R}}|C^kE^j\varphi(y)|,\quad k,j\in\mathbb{N}_0,
	\end{equation*}
	where $C\varphi:= (\partial(\varphi\circ\sigma))\circ\sigma ^{-1}$ and $E\varphi:=(\boldsymbol m(\varphi\circ \sigma))\circ \sigma^{-1} = \sigma^{-1}\cdot \varphi$. The operator $C$ can be rewritten as
	\begin{equation*}
		C\varphi(y) = \bigg(1 + \frac{2}{(y+\sqrt{y^2+4})^2}\bigg)\varphi'(y)=:\psi(y)\cdot \varphi'(y),\quad \varphi\in\mathcal S_\oplus(\mathbb{R}),y\in\mathbb{R}.
	\end{equation*}
	Because $\sigma^{-1},\psi\in\mathcal O_{\rm M}(\mathbb{R})$, both $C$ and $E$ have extensions in $\mathcal L(\mathcal S(\mathbb{R}))$. Thus $\mathcal S_\oplus(\mathbb{R}) = \mathcal S(\mathbb{R})$ and finally $\mathcal S_+(\mathbb{R}) = \mathcal S(\mathbb{R}^+)$.
\end{proof}

The most important property of $\mathcal S(\mathbb{R}^\pm)$ (next to being a closed subspace of $\mathcal S(\mathbb{R})$) is stated in the following corollary.

\begin{corollary}
	The map $x\mapsto |x|^v$ is in $\mathcal O_\mathrm{M}(\mathbb{R}^\pm)$ for each $v\in\mathbb{R}$.
\end{corollary}
\begin{proof}
	The continuity $\frac{1}{\boldsymbol m} $ was already shown in the proof to the last lemma with inequalities \eqref{eq:proof1overx_cont}. Of course $\boldsymbol m\varphi(x):=x\varphi(x)$ defines a continuous operator on $\mathcal S(\mathbb{R}^\pm)$, as well. The derivatives of $x\mapsto |x|^v$ can be bounded by terms of the form $x\mapsto x^k$ for $k\in\mathbb Z$, which concludes the proof.
\end{proof}

We now find a characterisation for the functions in $\mathcal O_{\rm M}(\mathbb{R}^\pm\times M;E)$. This space will be of importance later on, when we examine the Fourier image of $\mathcal S(G)$ in further detail and when we want to discuss the integral formula for the Kohn-Nirenberg quantization.

\begin{corollary}\label{corollary:char_slowly_increasing_on_Rtimestimes_M}
	A smooth function $f\colon \mathbb{R}^\pm\times M\to E$ is in $\mathcal O_{\rm M}(\mathbb{R}_\pm\times M; E)$, iff for each $k\in\mathbb{N}_0$, each $P\in\mathrm{Diff}_\mathcal{P}(M)$ and each continuous seminorm $p$ on $E$, there exists an $l\in\mathbb{N}$ and an $q\in\mathcal P(M)$, such that $p(\partial_\lambda^k P_x f(\lambda,x))\le (1 + |\lambda|^l + |\lambda|^{-l})q(x)$.
\end{corollary}
\begin{proof}
	We know, that $\mathcal O_{\rm M}(\mathbb{R}^+\times M;E)$ is the space of all smooth functions $f$ on $\mathbb{R}^+$, such that
	\begin{equation*}
		[\varphi \mapsto f\cdot \varphi] \in\mathcal L(\mathcal S(\mathbb{R}^\pm\times M),\mathcal S(\mathbb{R}^\pm\times M;E)).
	\end{equation*}
	We prove the statement for $\mathbb{R}^+$, then the other statement follows at once, since $\mathbb{R}^-$ is isomorphic to $\mathbb{R}^+$ by $x\mapsto -x$. Also, it is enough to consider $M=\mathbb{R}^n$, as the more general case follows by just using polynomial coordinate charts.
	
	Suppose $f\in\mathcal O_{\rm M}(\mathbb{R}^+\times \mathbb{R}^n;E)$. Because $f$ induces a continuous multiplication operator and because $\mathcal S(\mathbb{R}^+)$ is a subspace of $\mathcal S(\mathbb{R})$, for each $k\in\mathbb{N}_0$, $\alpha\in\mathbb{N}_0^n$ and each continuous seminorm $p$ on $E$, there is some $m\in\mathbb{N}$ and $C>0$ with
	\begin{equation*}
	\begin{split}
		\sup_{\lambda\in\mathbb{R}^+, x\in M} p(\partial_\lambda^k &\partial^\alpha_x(f(\lambda,x)\varphi(\lambda, x))) \\
		&\le C\max_{|\beta|,l\le m} \sup_{\lambda\in\mathbb{R}^+, x\in M} (1 + |\lambda|^m)(1+|x|^2)^m|\partial_\lambda^l \partial^\beta_x\varphi(\lambda,x)|,
		\end{split}
	\end{equation*}
	for all $\varphi\in\mathcal S(\mathbb{R}^+\times \mathbb R^n)$. We choose some $\varphi \in \mathcal S(\mathbb{R}^+\times\mathbb{R}^n)$, such that $\varphi\equiv 1$ on some neighbourhood around $(\lambda,x)=(1,0)$, and define $\varphi_{a,y}(x):=\varphi(xa^{-1},x-y)$ for $a>0$, $y\in\mathbb{R}^n$. Then
	\begin{equation*}
	\begin{split}
		 p(\partial^{(k,\alpha)}f(a,y)) &= p(\partial_\lambda^k \partial^\alpha_x(f(\lambda,x)\varphi_{a,y}(\lambda, x)))\big|_{(\lambda,x)=(a,y)} \\
		 &\le C\max_{|\beta|,l\le m} \sup_{\lambda\in\mathbb{R}^+, x\in \mathbb R^n} (1 + |\lambda|^m)(1+|x|^2)^m|\partial_\lambda^l \partial^\beta_x\varphi_{a,y}(\lambda,x)| \\
		&=C\max_{|\beta|,l\le m} \sup_{\lambda\in\mathbb{R}^+, x\in\mathbb R^n } a^{-l}(1 + |a\lambda|^m)(1+|x+y|^2)^m|\partial_\lambda^l \partial^\beta_x\varphi(\lambda,x)|  \\
		&\le C'(1 + a^m + a^{-m})(1+|y|^2)^m,
	\end{split}
	\end{equation*}
	where $k$, $\alpha$, $m$ and $C$ are as above. Of course this implies, that for each $k\in\mathbb{N}_0$, $P\in\mathrm{Diff}_\mathcal{P}(\mathbb{R}^n)$ and each continuous seminorm $p$ on $E$, there exists an $l\in\mathbb{N}$ and a $q\in\mathcal P(\mathbb{R}^n)$, such that
	\begin{equation}\label{eq:lemma_OMofRtimes_temp}
		p(\partial_\lambda^k P_x f(\lambda,x))\le (1 + |\lambda|^l + |\lambda|^{-l})q(x).
	\end{equation}
	
	Now for the converse implication. Let $f\colon \mathbb{R}^+\times\mathbb{R}^n \to \mathbb{C}$ be any smooth function, such that for $p$, $k$ and $P$ we find $m$ and $q$ for the inequality \eqref{eq:lemma_OMofRtimes_temp}. Then for arbitrary $\varphi\in\mathcal S(\mathbb{R}^+\times M)$,
	\begin{equation*}
	\begin{split}
		\sup_{\lambda\in\mathbb{R}^+, x\in \mathbb{R}^n} (1 + |\lambda|^k)&(1+|x|^2)^k p(\partial^\alpha(f\,\varphi)(\lambda,x)) \\
		&\le C \sup_{\lambda\in\mathbb{R}^+, x\in \mathbb{R}^n}(1 + \lambda^k)(1 + |x|^2)^k\sum_{\beta \le\alpha}|\partial^{\alpha-\beta}f(\lambda,x)\,\partial^\beta\varphi(\lambda,x)| \\
		&\le C' \sup_{x\in\mathbb{R}_+}(1 + |x|^2)^{k+m}(1 + \lambda^{k+m} + \lambda^{k-m})\sum_{\beta\le \alpha}|\partial^\beta\varphi^{(j)}(x)|.
	\end{split}
	\end{equation*}
	Since $\frac{1}{\boldsymbol m}$ is a continuous operator on $\mathcal S(\mathbb{R}^+)$, the last line defines a continuous seminorm on $\mathcal S(\mathbb{R}^+\times \mathbb{R}^n)$. Thus the operator $\varphi\mapsto f\cdot \varphi$ is continuous.
\end{proof}

From the polynomial structures on $\mathbb{R}^+$ and $\mathbb{R}^-$, we construct the polynomial manifold $\mathbb{R}^\times = \mathbb{R}^+\,\dot\cup\,\mathbb{R}^-$. Its Schwartz space $\mathcal S(\mathbb{R}^\times) = \mathcal S(\mathbb{R}^+)\oplus\mathcal S(\mathbb{R}^-)$ can be seen as the closed subspace of $\mathcal S(\mathbb{R})$ of functions $f$, which vanish of arbitrary order in $0$, i.e. $\partial^k f(0) = 0$ for all $k\in\mathbb{N}_0$. The dual space and the Fourier image of $\mathcal S(\mathbb{R}^\times)$ will play an important role in the coming discussion. The first statement requires no further proof.

\begin{lemma}\label{lemma:fouriertransform_on_homogenuous_Schartz_spaces}
	The image of $\mathcal S(\mathbb{R}^\times)$ under the Fourier transform on $\mathbb{R}$, $\mathcal F_\mathbb{R}$, is $\mathcal S_*(\mathbb{R})$, which is defined to be the subspace of Schwartz functions $f$ with vanishing moments of arbitrary order, i.e.
	\begin{equation*}
		\int_\mathbb{R} f(x)\, p(x)\,\mathrm{d} x=0,\quad\text{for all}\quad p\in\mathcal P(\mathbb{R}).
	\end{equation*}
\end{lemma}

The next Lemma is less obvious. It is an extension of the well know fact, that $\mathcal S_*'(\mathbb{R})$, as a vector space, can be identified with the quotient $\mathcal S'(\mathbb{R})/\mathcal P(\mathbb{R})$ e.g. \cite[Proposition 1.1.3]{grafakos}.

\begin{lemma}\label{lemma:char_dual_S0}
	Let $E$ be a nuclear Fr\'echet space and $\mathcal E'_0(\mathbb{R})$ the space of distributions on $\mathbb{R}$ with support in $\{0\}$. Then
	\begin{equation*}
		(\mathcal S(\mathbb{R}^\times)\,\hat\otimes\, E )'\simeq (\mathcal S'(\mathbb{R})\,\hat\otimes\, E')/(\mathcal E'_0(\mathbb{R})\otimes E'),
	\end{equation*}
	especially $\mathcal E'_0(\mathbb{R})\otimes E'$ is a closed subspace of $\mathcal S'(\mathbb{R})\,\hat\otimes\, E'$.
\end{lemma}
\begin{proof}
	First we will prove, that $Z:=\mathcal E'_0(\mathbb{R})\otimes E'$ is a closed subspace of $X'\simeq\mathcal S'(\mathbb{R})\,\hat\otimes\, E'$, where $X:=\mathcal S(\mathbb{R})\,\hat\otimes\, E$. The family $(\partial^k\delta_0)_{k\in\mathbb{N}_0}$ is a basis for $\mathcal E'_0(\mathbb{R})$ where $\delta_0$ is the delta distribution. We use Lemma \ref{lemma:allg_tenserop_stetig_alt} on the sequence $P_N$ of projections onto the subspaces spanned by $\{\delta_0,\dotsc,\partial^N\delta_0\}$ and conclude, that $Z$ is sequentially dense in its closure $\overline{Z}$. Furthermore we realize, that for any $\phi\in \overline{Z}$ there is a sequence $(e'_k)\subset E'$, such that
	\begin{equation*}
		\phi = \lim_{N\to\infty}\phi_N := \lim_{N\to\infty} \sum_{k=0}^N (\partial^k\delta_0)\otimes e'_k.
	\end{equation*}
	Because $X$ is a Fr\'echet space and $Z\subset X'$, we can apply the Banach-Steinhaus Theorem. Hence there exists a continuous seminorm $q$ on $E$ and $M\in\mathbb{N}$, such that
	\begin{equation*}
		|\phi_N(f)| \le \max_{k\le M}\sup_{x\in\mathbb{R}}\langle x\rangle^M q(\partial^k_x f(x))
	\end{equation*}
	for all functions $f\in X=\mathcal S(\mathbb{R})\,\hat\otimes\, E$ and all $N\in\mathbb{N}$.
	
	Now suppose there is one $l>M$, such that $e'_l\neq 0$. Let us define the sequence of Schwartz functions $f_m(x):= \mathrm{e}^{\mathrm i m x} \psi(x) e/m^{l-1}$, where $\psi$ is a Schwartz function equal to one near zero and $e\in E$ with $e'_l (e)= 1$. We arrive at
	\begin{equation*}
		|\phi_l(f_m)| = \bigg|\sum_{k=0}^l \frac{(\mathrm i m)^k}{m^{(l-1)}} e'_k(e)\bigg|\xrightarrow{m\to\infty} \infty.
	\end{equation*}
	But also
	\begin{equation*}
		\sup_{m\in\mathbb{N}}\max_{k\le M}\sup_{x\in\mathbb{R}}\langle x\rangle^M q(\partial^k_x f_m(x))<\infty,
	\end{equation*}
	which is a contradiction. Hence $\phi\in Z$, i.e. $\phi$ is in the finite span of the $\partial^k\delta$ and $e'_k$.
	
	Now let $Y:= Z^\circ$ be the polar of $Z$. Because $X$ is reflexive, we may identify $Y\subset X$. Since $Z$ is a closed subspace, we also have $Y^\circ = Z^\circ{}^\circ= Z $. Since $\partial^k\delta_0\otimes e' \in Z$ for all $k\in\mathbb{N}_0$, $e'\in E'$ and
	\begin{equation*}
		(\partial^k\delta_0\otimes e')(\varphi) = e'(\partial^k\varphi(0)),\quad \text{for}\quad \varphi\in X = \mathcal S(\mathbb{R};E),
	\end{equation*}
	it is quite obvious, that $Y=\mathcal S(\mathbb{R}^\times)\,\hat\otimes\, E$.
	
	Since $E$ is a nuclear Fr\'echet space, $X$ is a nuclear Fr\'echet space. That also means, that $X$ is an (FS) space. I.e. it is the projective limit
	\begin{equation*}
		X_1\leftarrow X_2\leftarrow \dots \leftarrow X
	\end{equation*}
	of a sequence of Banach spaces $(X_k)_k$ with compact maps $X_k\leftarrow X_{k+1}$ \cite[Chapter 3, Corollary 3 to Theorem 7.3]{schaefer_tvs}. Notice that the maps $X_k\leftarrow X_{k+1}$ are weakly compact, too. Now we may conclude the proof, by using Theorem 13 of \cite{komatsu_sequences_of_bspaces}. The theorem states, that in our situation -- $Y$ is closed and $X$ is an (FS) space -- we have $Y'\simeq X'/Y^\circ$.
\end{proof}

By using the euclidean Fourier transform in combination with the last lemma, we get the following corollary.

\begin{corollary}\label{corollary:char_dual_Sstar}
	Let $E$ be a nuclear Fr\'echet space, then
	\begin{equation*}
		(\mathcal S_*(\mathbb{R})\,\hat\otimes\, E)' \simeq (\mathcal S'(\mathbb{R}) \,\hat\otimes\, E')/(\mathcal P(\mathbb{R})\otimes E')
	\end{equation*}
	and $\mathcal P(\mathbb{R})\otimes E'$ is closed in $\mathcal S'(\mathbb{R}) \,\hat\otimes\, E'$.
\end{corollary}

Furthermore, this characterization for the dual spaces of $\mathcal S(\mathbb{R}^\times)\,\hat\otimes\, E$ and $\mathcal S_*(\mathbb{R})\,\hat\otimes\, E$ by quotient spaces, enables us to find subspaces of $\mathcal S'(\mathbb{R})\,\hat\otimes\, E'$ which are embedded into these dual spaces. Suppose $F$ is a Banach space, such that there is a continuous embedding $E\hookrightarrow F$ with dense range. Then we may see, that the Lebesgue-Bochner spaces $L^p(\mathbb{R}; F')$ are embedded into $\mathcal S_*'(\mathbb{R})\,\hat\otimes\, E'$ and into $\mathcal S(\mathbb{R}^\times)\,\hat\otimes\, E'$ for $p\in (1,\infty)$. Here we define the distribution corresponding to $f\in L^p(\mathbb{R}; F')$ by
\begin{equation*}
	T_f(\varphi) := \int_\mathbb{R} \langle f(x),\varphi(x)\rangle\,\mathrm d x,\quad \varphi\in\mathcal S(\mathbb{R}; E),
\end{equation*}
where $\langle\cdot,\cdot\rangle$ denotes the dual pairing on $F'\times F$. Notice, that $T_f$ is indeed an injective map into $\mathcal S'(\mathbb R;E')$, since $f=0$ almost everywhere, iff $T_f(\varphi\otimes e)=0$ for all $\varphi\in\mathcal S(\mathbb R)$ and all $e\in E$.

Though, we can make a much more general claim. For this purpose, we define the following subspaces of $\mathcal S'(\mathbb{R})\,\hat\otimes\, E'$.
\begin{equation*}
\begin{split}
	\dot{\mathcal B}'(\mathbb{R};E')&:=\{\phi \in \mathcal S'(\mathbb{R})\,\hat\otimes\, E' \mid \forall_{\varphi\in \mathcal S(\mathbb{R})\,\hat\otimes\, E}\ \phi(\varphi(\cdot -x))\xrightarrow{|x|\to\infty} 0\} \\
	\widetilde{\mathcal B}'(\mathbb{R};E')&:= \{\phi \in \mathcal S'(\mathbb{R})\,\hat\otimes\, E' \mid \forall_{\varphi\in \mathcal S(\mathbb{R})\,\hat\otimes\, E}\ \phi(\varphi(\lambda^{-1}\cdot ))\xrightarrow{\lambda\to 0} 0\}
\end{split}
\end{equation*}
\begin{lemma}\label{lemma:lebesgue_bochner}
	Let $F$ be a Banach space as described above. The Lebesgue-Bochner space $L^p(\mathbb{R}; F')$ is a subspace of $\dot{\mathcal B}'(\mathbb{R};E')$ for $p\in [1,\infty)$ and a subspace of $\widetilde{\mathcal B}'(\mathbb{R};E')$ for $p\in [1,\infty]$ with respect to the embedding $f\mapsto T_f$.
\end{lemma}
\begin{proof}
	Let $f\in L^p(\mathbb R;F')$ and let $\varphi \in\mathcal S(\mathbb R)\,\hat\otimes\, E$ then also
	\begin{equation*}
		[x\mapsto (1+x^2)\,\varphi(x)] \in L^q(\mathrm{R};F)
	\end{equation*}
	for each $1 = 1/p + 1/q$. Suppose first $p\in [1,\infty)$, then for some $C>0$ independent of $x\in\mathbb R$
	\begin{equation*}
		|T_f(\varphi(\cdot -x))| \le \int_\mathbb{R} |\langle f(y),\varphi(y-x)\rangle|\,\mathrm{d}y \le C \left(\int_\mathbb{R} \frac{\|f(y)\|_F^p}{(1+(x-y)^2)^p}\,\mathrm{d}y\right)^{\frac{1}{p}}.
	\end{equation*}
	Now let $\varepsilon>0$ be arbitrary and let $R>0$ be big enough, such that
	\begin{equation*}
		\int\limits_{\{y\in\mathbb R\colon  |y|\ge R\}} \|f(y)\|_{F'}^p\,\mathrm{d}y \le \varepsilon,
	\end{equation*}
	With this inequality, we get
	\begin{equation*}
		\left(\int_\mathbb{R} \frac{\|f(y)\|_{F'}^p}{(1+(x-y)^2)^p}\,\mathrm{d}y\right)^{\frac{1}{p}}\le \left(\varepsilon + \int\limits_{\{y\in\mathbb R\colon |y|\le R\}} \frac{\|f(y)\|_{F'}^p}{(1+(x-y)^2)^p}\,\mathrm{d}y\right)^{\frac{1}{p}}\xrightarrow{x\to\pm\infty} \varepsilon^\frac{1}{p}.
	\end{equation*}
	Hence $T_f \in \dot{\mathcal B}'(\mathbb{R};E')$, because $\varepsilon>0$ can be arbitrarily small. With the same calculation as before, we get
	\begin{equation*}
		|T_f(\varphi(\cdot/\lambda))| \le C\left(\varepsilon + \int\limits_{\{y\in\mathbb R\colon |y|\le R\}} \frac{\|f(y)\|_{F'}^p}{(1+(y/\lambda)^2)^p}\,\mathrm{d}y\right)^{\frac{1}{p}}\xrightarrow{\lambda\to 0} C\varepsilon^\frac{1}{p}.
	\end{equation*}
	Thus $T_f\in\widetilde{\mathcal B}'(\mathbb{R};E')$. Now suppose $p=\infty$. Here we have
	\begin{equation*}
		|T_f(\varphi(\cdot/\lambda))| \le \lambda\, \operatorname*{ess\,sup}_{x\in\mathbb R}\|f(x)\|_{F'}\,\int_{\mathbb R} \|\varphi(y)\|\,\mathrm{d}y \xrightarrow{\lambda\to 0} 0.
	\end{equation*}
	Hence also $T_f \in\widetilde{\mathcal B}'(\mathbb{R};E')$ for this case.
\end{proof}

Note, that the distributions in $\dot{\mathcal B}'(\mathbb{R};E')$ can have any form in a bounded region, whereas distributions in $\widetilde{\mathcal B}'(\mathbb{R};E')$ can have any form away from zero, as long as they are tempered.

\begin{proposition}\label{prop:distr_conv_to_zero_embedding}
	The quotient maps
	\begin{equation*}
	\begin{split}
		\mathcal S'(\mathbb{R})\,\hat\otimes\, E' &\to \mathcal{S}'(\mathbb{R}^\times)\,\hat\otimes\, E',\\
		\mathcal S'(\mathbb{R})\,\hat\otimes\, E' &\to \mathcal S_*'(\mathbb{R}) \,\hat\otimes\, E',
	\end{split}
	\end{equation*}
	restrict to embeddings
	\begin{equation*}
	\begin{split}
		\widetilde{\mathcal B}'(\mathbb{R};E') &\hookrightarrow \mathcal{S}'(\mathbb{R}^\times)\,\hat\otimes\, E', \\
		\dot{\mathcal B}'(\mathbb{R};E') &\hookrightarrow \mathcal S_*'(\mathbb{R}) \,\hat\otimes\, E'.
	\end{split}
	\end{equation*}
\end{proposition}
\begin{proof}
	A short calculation yields
	\begin{equation*}
		\widetilde{\mathcal B}'(\mathbb{R};E')\cap \mathcal E'_0(\mathbb{R})\otimes E' = \{0\} = \dot{\mathcal B}'(\mathbb{R};E') \cap \mathcal P(\mathbb{R}) \otimes E'.
	\end{equation*}
	Together with the above lemma and corollary, this already concludes the proof.
\end{proof}

\subsection{Flat orbits of Homogeneous Lie groups}

Let $\operatorname{Ad}$ be the adjoint action of $G$ on $\mathfrak g$. Denote by $\operatorname{Ca}_x \xi:=\xi\circ \operatorname{Ad}_{x^{-1}}$ the coadjoint action of $x\in G$ on linear functionals $\xi\in\mathfrak g'$. A subalgebra $\mathfrak{m}\subset \mathfrak g$ is called polarizing to $\ell\in\mathfrak{g}'$, iff $\ell([\mathfrak{m},\mathfrak{m}])=\{0\}$ and $\mathfrak{m}$ is a maximal algebra fulfilling this condition. For any $\xi\in\mathfrak{g}'$ we can find at least one polarizing algebra. There is a bijection between the coadjoint Orbits and the irreducible unitary representations of $G$. It can be described by $[\pi] \leftrightarrow \Omega = \operatorname{Ca}_G \xi$, where $\pi$ is unitarily equivalent to the induced representation of $\chi( m) = \mathrm{e}^{2\pi\mathrm{i} \xi(m)}$ for $m\in \mathfrak{ m}\subset G$ for some maximal subordinate algebra $\mathfrak{m}$ of $\ell$ \cite[Theorems 2.2.1 - 2.2.4]{rep_nilpotent_lie_groups}. This correspondence only depends on the orbit $\Omega$ and not on the choice of element $\xi$ spanning $\Omega$ or the choice of polarizing algebra $\mathfrak{m}$. We will write $\pi\sim \xi$ or $\pi\sim \Omega$, if the equivalence class of $\pi$ corresponds to the orbit $\Omega = \operatorname{Ca}_G\xi$. We equip $\widehat G$ with the initial topology with respect to the bijection $[\pi]\mapsto \Omega$ for $\pi\sim\Omega$ from $\widehat G$ to $\mathfrak{g}'/G$. For any $\xi$ the orbit $\Omega = \operatorname{Ca}_G\xi$ is an even dimensional polynomial manifold \cite[page 521]{pedersen_geometric_quant} and \cite[Lemma 1.3.2]{rep_nilpotent_lie_groups}.

A Jordan-H\"older basis of $\mathfrak{g}$, is a basis $(e_j)_j$, such that the linear hull $\mathfrak g_k = \operatorname{span} \{e_1,\dotsc e_k\}$, is an ideal in $\mathfrak{g}$ for each $k\le\dim G$. Let $q_k$ be the quotient map $\mathfrak{g}'\to\mathfrak{g}'/\mathfrak{g}_k^\circ$. The set of jump indices $J$ is the set of $j>1$, such that
\begin{equation*}
	\dim q_j (\Omega)-\dim q_{j-1}(\Omega) = 1
\end{equation*}
Let us denote $\mathfrak{g}_J:=\operatorname{span}\{e_j\mid j\in J\}$. From Corollary 3.1.5 of \cite{rep_nilpotent_lie_groups} follows, that a polynomial chart of $\Omega$ is given by
\begin{equation*}
	\sigma_\Omega \colon \Omega \to \mathfrak{g}_J' \colon \xi \mapsto \xi\restriction_{\mathfrak{g}_J}.
\end{equation*}
This equivalence between orbits and the corresponding subspaces $\mathfrak{g}_J$, leads to the definition of the orbital Fourier transform as the integral
\begin{equation*}
	\mathcal F_\Omega\varphi(x):=\int_\Omega \mathrm{e}^{-2\pi\mathrm{i} \xi (x)} \varphi (\xi) \,\mathrm{d}\theta_\Omega (\xi), \quad x\in\mathfrak{g}_J,\ \varphi\in\mathcal S(\Omega),
\end{equation*}
where $\theta_\Omega\circ \sigma^ {-1}_\Omega$ is a Haar measure on $\mathfrak{g}_J'$. The Pedersen quantization \cite{pedersen_matrix_coefficients} is the equivalent of the Weyl quantization for general simply connected nilpotent Lie groups. It is defined by the integral
\begin{equation*}
	\operatorname{op}_\pi(\varphi) := \int_{\mathfrak{g}_J} \pi(x) \int_{\Omega} \mathrm{e}^{-2\pi\mathrm{i}\xi(x)}\varphi(\xi) \,\mathrm{d}\theta_\Omega(\xi)\,\mathrm{d}\nu_\Omega(x),
\end{equation*}
for some representation $\pi\sim\Omega$ and a fitting Haar measure $\nu_\Omega$ on $\mathfrak{g}_J$. We can easily see, that the outermost integral converges in $\mathcal L(H_\pi)$. The following theorem fixes the choice of $\nu_\Omega$.
\begin{theorem}\label{lemma:orbital_pedersen_quant_gelfand_triple}
	For each $\theta_\Omega$ as above, there is a unique $\nu_\Omega$, such that the Pedersen quantization to $\pi\sim\Omega$ extends to a Gelfand triple isomorphism 
	\begin{equation*}
		\operatorname{op}_\pi \colon \mathcal G(\Omega,\theta_\Omega) \to \mathcal G_\mathrm{op}(\pi).
	\end{equation*}
\end{theorem}
\begin{proof}
	This is essentially stated in \cite[Theorem 4.1.4]{pedersen_matrix_coefficients}. Here Pedersen proves, that
	\begin{equation*}
		\mathcal S(\Omega) \to \mathcal B(H_\pi)_\infty\colon a\mapsto \operatorname{op}_\pi (a)
	\end{equation*}
	is a homeomorphism, where $\mathcal B(H_\pi)_\infty$ is the space of smooth operators with respect to $\pi$. The spaces of smooth operators is defined to be $\mathcal B(H_\pi)_\infty= H_\Pi^\infty$, where $\Pi$ is the unitary representation of $G\times G$ on $\mathcal{H\!S}(H_\pi)$ defined by $\Pi(x,y) T = \pi(x)\circ T \circ \pi(y)^{-1}$. Furthermore Pedersen shows that
	\begin{equation*}
		\int_\Omega a\,\overline b\,\mathrm{d}\theta_\Omega = \operatorname{Tr}[\operatorname{op}_\pi (a)\,\operatorname{op}_\pi (b)^*],\quad\text{for }a,b\in\mathcal S(\Omega)
	\end{equation*}
	for a suitable choice of $\nu_\Omega$.
	
	In order to fit this result in our scheme we will make sure, that $\mathcal L(H^{-\infty}_\pi,H^\infty_\pi) = \mathcal B(H_\pi)_\infty$ as topological vector spaces. It is easy to see, that this identity holds in the sense that each $T\mapsto T\circ \mathcal I$ is a bijection from the left-hand side to the right-hand side, where $\mathcal I\colon H_\pi\hookrightarrow H^{-\infty}_\pi$ is the embedding defined by the real structure on $\mathcal G(\pi)$.
	
	It is left to check that the topologies on both sides coincide. For $P\in\mathfrak{u}(\mathfrak{g}_\mathrm{L})$ denote by $\pi(P)_*\in\mathcal L(H^{-\infty}_\pi)$ the unique operator fulfilling $\mathcal I\pi(P)v = \pi(P)_* \mathcal I v$ for $v\in H^\infty_\pi$. By \cite{diff_vectors} there is $P\in\mathfrak{u}(\mathfrak{g}_\mathrm{L})$, such that $\pi(P)$ is invertible on $H^\infty_\pi$, $\pi(P)^{-1}$ can be extended to a nuclear operator on $H_\pi$ and $H_\pi^{-\infty}$ is the compact inductive limit of the Hilbert spaces $H^{-k}_\pi:= \pi(P)_*^k \mathcal I H_\pi$, equipped with the norm $\|w\|_{-k}:=\|\mathcal I^{-1}\pi(P)^{-k} w\|_{H_\pi}$. But this means each bounded $B\subset H^{-\infty}_\pi$ is in fact a bounded set in some $H^{-k}_\pi$, i.e. $B= \pi(P)^k_*\mathcal I\tilde B$ for $\tilde B$ bounded in $H_\pi$.
	The topology in $\mathcal L(H^{-\infty}_\pi,H^\infty_\pi)$ is defined by the seminorms
	\begin{equation*}
		T\mapsto \sup_{\phi\in B} \|\pi(P)^j T\phi\|\,,\quad\text{for}\quad B\subset H^{-\infty}_\pi\text{ bounded and }j\in\mathbb N_0.
	\end{equation*}
	The above implies, that actually it is enough to consider the seminorms
	\begin{equation*}
		T\mapsto \|\pi(P)^j\, T\,\mathcal I\, \pi(P)^k\|\,,\quad\text{for }j,k\in\mathbb N_0.
	\end{equation*}
	Now the above seminorms are also continuous on $\mathcal B(H_\pi)_\infty$, so by the open mapping theorem for Fr\'echet spaces, $\mathcal L(H^{-\infty}_\pi,H^\infty_\pi) = \mathcal B(H_\pi)_\infty$ as topological vector spaces and
	\begin{equation*}
		\mathcal S(\Omega) \to \mathcal L(H_\pi^{-\infty},H_\pi^\infty)_\infty\colon a\mapsto \operatorname{op}_\pi (a)
	\end{equation*}
	is a homeomorphism.	
	
	Finally by the dense and continuous embeddings
	\begin{equation*}
		\mathcal S(\Omega)\hookrightarrow L^2(\Omega,\theta_\Omega)\quad\text{and}\quad \mathcal L(H^{-\infty}_\pi,H^\infty_\pi)\hookrightarrow \mathcal{H\!S}(H_\pi)
	\end{equation*}
	we may extend $\operatorname{op}_\pi$ to a unitary operator between $L^2(\Omega,\theta_\Omega)$ and $\mathcal{H\!S}(H_\pi)$, and as such, even to a Gelfand triple isomorphism
	\begin{equation*}
		\operatorname{op}_\pi \colon \mathcal G(\Omega,\theta_\Omega)\to \mathcal G_\mathrm{op}(\pi).
	\end{equation*}
	
	Note, that Pedersen uses the convention $\xi \leftrightarrow \chi(\cdot) = \mathrm{e}^{\mathrm{i}\xi(\cdot)}$ for bijection between functionals and characters. Though adjusting the formulas just results in additional constants, that may be hidden away inside the measures $\nu_\Omega$ and $\theta_\Omega$.
\end{proof}
Though in our case, we can simplify this process by a lot, since we are only interested in representations derived from generic resp.\ flat orbits.

An orbit is called \emph{generic}, if for each $k$ the dimension of $q_k(\Omega)$ is maximal compared to all other orbits. Let us denote the set of equivalence classes derived from generic orbits by $\widehat G_\mathrm{gen}\subset \widehat G$. Note, that the Plancherel measure $\widehat \mu$ is concentrated on $\widehat G_\mathrm{gen}$.

A representation $\pi\in\mathrm{Irr}(G)$ is \emph{square integrable modulo the center}, if $x\mapsto |(\pi(x)v,w)_{H_\pi}|$ is square integrable on $\mathfrak{g}/\mathfrak{z}$ with respect to the Haar measure for all $v,w\in H_\pi$. Let us denote the set of irreducible representations, that are square integrable modulo the center, by $\mathrm{SI/Z}(G)\subset\mathrm{Irr}(G)$ and pairs of such representations together with some matching real structure by $\mathrm{SI/Z}_\mathbb{R}(G)$. Suppose $\pi\sim\Omega = \operatorname{Ca}_G\xi$, then $\pi\in\mathrm{SI/Z}(G)$, if and only if $\Omega = \xi + \mathfrak{z}^\circ$  \cite{Moore_Wolf}. Furthermore, if $\mathrm{SI/Z}(G)\neq \emptyset$, then the orbits to representations in $\mathrm{SI/Z}(G)$ are exactly those having the maximal possible dimension \cite[Corollary 4.5.6]{rep_nilpotent_lie_groups}. Also, the jump indices for $\pi\in\mathrm{SI/Z}(G)$ are given by $J=\{k+1,k+2,\dotsc,\dim G\}$, where $k=\dim\mathfrak{z}$, and the equivalence class $[\pi]\in\widehat G$ is uniquely determined by the central character $\pi\restriction_{\mathfrak{z}} = e^{2\pi\mathrm{i} \xi(\cdot)}\operatorname{id}_{H_\pi}$, where $\xi\in\omega^\circ\simeq \mathfrak{z}'$. For this fact see \cite[Corollaries 4.5.3 and 4.5.4]{rep_nilpotent_lie_groups}.

For all $\pi\in\mathrm{SI/Z}(G)$ the Pedersen quantization is simpler, for we can just take \emph{one} Haar measure $\theta$ on $\mathfrak{z}^\circ$ and translate it to a measure $\theta_\Omega$ on $\Omega\sim\pi$ for each $\pi\in\mathrm{SI/Z}(G)$. The subspace $\omega:=\mathfrak{g}_J$ complements $\mathfrak{z}$ in $\mathfrak{g}$ and is the same for each representation in $\mathrm{SI/Z}(G)$. We get a Gelfand triple isomorphisms
\begin{equation*}
	\mathcal G(\mathfrak{z}^\circ,\theta)\to \mathcal G(\Omega,\theta_\Omega)\colon \phi\mapsto \phi\circ P_{\mathfrak{z}^\circ},
\end{equation*}
where $P_{\mathfrak{z}^\circ}$ is the projection onto $\mathfrak{z}^\circ$ along $\omega^\circ$. Using this isomorphism, we adjust the Pedersen quantization.

\begin{definition}\label{definition:pedersen_quant_4_flat_orbit}
	We will use the Pedersen quantization $\mathfrak{op}_\pi$ on $\mathcal G(\mathfrak{z}^\circ,\theta)$ with respect to $\pi\in\mathrm{SI/Z}(G)$, defined by
	\begin{equation*}
		\mathfrak{op}_\pi\colon \mathcal G(\mathfrak{z}^\circ,\theta)\to \mathcal G_\mathrm{op}(\pi),\  \phi\mapsto \operatorname{op}_\pi(\phi\circ P_{\mathfrak{z}^\circ}).
	\end{equation*}
\end{definition}

This version of Pedersen quantization takes on the form
\begin{equation*}
	\mathfrak{op}_\pi(\varphi) = \int_\omega \pi(x) \int_{\mathfrak{z}^\circ} \mathrm{e}^{-2\pi\mathrm{i} \xi(x)} \varphi(\xi)\,\mathrm{d}\theta(\xi)\,\mathrm{d}\nu(x),
\end{equation*}
where $\nu=\nu_\Omega$ depends on $\theta$. Of course $\mathfrak{op}_\pi$ is a Gelfand triple isomorphism, as well.

Now we will discuss the concept of generic orbits and square integrable (modulo the center) representation in context with homogeneous groups. The Lie group $G=\mathfrak{g}$ is called a homogeneous Lie group, if it is equipped with a group of dilations
\begin{equation*}
	(0,\infty) \to \mathrm{Hom}(G) \colon \lambda\mapsto\delta_\lambda,
\end{equation*}
where $\delta_\lambda x = \mathrm e^{\log(\lambda)A}x$ is also a Lie algebra isomorphism and $A$ is a diagonalizable map with positive eigenvalues. The number $Q:=\operatorname{Tr}[A]$ is the \emph{homogeneous dimension} of $G$.

We may always decompose $\mathfrak g$ into eigenspaces $\mathcal E_\kappa$ of $A$ to Eigenvalues $\kappa> 0$, i.e.
\begin{equation*}
	\mathfrak g = \bigoplus_{\kappa> 0}\mathcal E_\kappa,\quad\text{where}\quad [\mathcal E_\kappa,\mathcal E_{\kappa'}]\subset \mathcal E_{\kappa+\kappa'}.
\end{equation*}
Notice that the center $\mathfrak{z}$ of $\mathfrak{g}$ is always an eigenspace to both $\delta_\lambda$ and $A$, since
\begin{equation*}
	[\delta_\lambda z, x] = \delta_{\lambda}[z,\delta_{\lambda^{-1}}x] = 0\,\quad\text{for all}\quad \lambda>0,z\in\mathfrak{z}\text{ and }x\in\mathfrak{g}.
\end{equation*}
For every $\mu>0$ the space $\bigoplus_{\kappa\ge\mu}\mathcal E_\kappa$ is an ideal in $\mathfrak g$. We may always choose a Jordan-H\"older basis $(e_j)_j$ through these ideals \cite[Theorem 1.1.13]{rep_nilpotent_lie_groups}.

If $\dim\mathfrak z=1$, then the center fulfils $\mathfrak z = \mathcal E_\mu$ for $\mu = \max\{\kappa>0\mid \mathcal E_\kappa\neq\{0\}\}$. Hence, one vector of our chosen Jordan-H\"older basis of eigenvectors will always lie in the center $\mathfrak z$. We also have the unique decomposition
\begin{equation*}
	\mathfrak{g} = \mathfrak{z}\oplus \omega, \quad \omega \text{ is } A\text{-invariant}.
\end{equation*}
Now for $\lambda < 0$ denote
\begin{equation*}
	\delta_\lambda x:= - \delta_{|\lambda|}x\text{ for }x \in\mathfrak z,\quad \text{and}\quad \delta_\lambda x := \delta_{|\lambda|}x\text{ for }x\in\omega.
\end{equation*}
Furthermore, let also $\delta_\lambda \xi := \xi\circ \delta_\lambda$ for $\lambda\in\mathbb{R}^\times$ and $\xi\in\mathfrak{g}'$.

The question arises whether generic orbits are mapped to generic orbits by $\delta_\lambda$. The dilation $\delta_\lambda$ on $\mathfrak{g}'/\mathfrak{g}_k^\circ$ is a well defined vector space isomorphism by $\delta_\lambda \circ q_j := q_j\circ \delta_\lambda$, since $\mathfrak g_k$ and thus also $\mathfrak{g}_k^\circ$ are $\delta_\lambda$-invariant. Furthermore
\begin{equation}\label{eq:dimension_generic_orbits_dilation}
	\dim q_j(\delta_\lambda\Omega) = \dim \delta_\lambda\circ q_j (\Omega) = \dim q_j(\Omega).
\end{equation}
Thus $\delta_\lambda \Omega$ is generic for each $\lambda\in\mathbb{R}^\times$.

Now take any $\pi\in \mathrm{Irr}_\mathbb{R}(G)$ with real structure $\mathcal C_\pi$ and define $\overline \pi:=\mathcal C_\pi\pi\mathcal C_\pi\in\mathrm{Irr}_\mathbb{R}(G)$ equipped with the same real structure. The representation $\overline \pi$ is equivalent to the dual representation of $\pi$. Now denote
\begin{equation*}
	\pi_\lambda (x):= \pi(\delta_\lambda x)\text{ for }\lambda> 0,\quad \text{and}\quad \pi_\lambda(x) := \overline\pi_{|\lambda|}(x) :=\overline\pi(\delta_{|\lambda|}g)\text{ for }\lambda<0.
\end{equation*}
All the representations $\pi_\lambda$ are irreducible unitary representations acting on $\mathcal H_\pi$ resp.\ acting smoothly on $\mathcal H_\pi^\infty$.
With these definitions and the discussion above, we get the equivalence of the three statements
\begin{itemize}
	\item $\pi\in\mathrm{SI/Z}(G)$, if and only if $\pi_\lambda \in\mathrm{SI/Z}(G)$,
	\item $[\pi]\in\widehat G_\mathrm{gen}$, if and only if $[\pi_\lambda]\in\widehat G_\mathrm{gen}$,
	\item $\pi\sim\xi$, if and only if $\pi_\lambda\sim\delta_\lambda\xi$.
\end{itemize}

Suppose that $\mathrm{SI/Z}(G)\neq \emptyset$ and $\dim \mathfrak{z}= 1$. Furthermore, suppose we chose a Jordan-H\"older basis  of eigenvectors to $A$. Let $\pi\in\mathrm{SI/Z}(G)$. As every equivalence class of representations in $\widehat G_\mathrm{gen}$ only depends on its central character, we get a bijection between $\mathbb{R}^\times$ and $\widehat G_\mathrm{gen}$. Hence, $\pi\in\mathrm{SI/Z}(G)$, if and only if $[\pi]\in\widehat G_\mathrm{gen}$.

We can even go one step further. The dilations $\delta_\lambda$ help us to understand $\widehat G$ as measure space. For this purpose we need the Pfaffian $P\!\! f(\xi)$ to a coadjoint orbit $\Omega = \operatorname{Ca}_G \xi$, which is defined by $P\!\!f(\xi)^2 = \det B_\xi$ up to a sign. Here $B_\xi:= (\xi([e_j,e_i]))_{j,i\in J}$ where the $(e_j)_{j\in J}$ span $\omega$.

We define $\kappa>0$ and $B\in\mathcal L(\omega)$, by $\delta_\lambda \eta := \operatorname{sgn}(\lambda)|\lambda|^{\kappa}\eta$, for $\eta\in\omega^\circ$ and $A|_\omega =B$.

\begin{proposition}\label{prop:dilatation_dual_group_measure_isomorph}
	Suppose $G$ is a homogeneous group, $\pi\in \mathrm{SI/Z}_\mathbb{R}(G)$ and $\dim \mathfrak{z}=1$, then
	\begin{equation*}
		(\mathbb{R}^\times,\kappa|\lambda|^{Q-1}|P\!\!f(\ell)|\,\mathrm{d}\lambda)\to (\widehat G_\mathrm{gen}, \widehat\mu) \colon \lambda \mapsto [\pi_\lambda],
	\end{equation*}
	where $\pi\sim\ell\in\omega^\circ$, is a homeomorphism resp.\ a strict isomorphism between the Borel measure spaces. Furthermore, if $\Omega$ is a fixed generic orbit, then $\lambda\mapsto \delta_\lambda\Omega$ defines a bijection between $\mathbb{R}^\times$ and the generic orbits.
\end{proposition}
\begin{proof}
	Let $U$ be the Zariski open set of functionals $\xi\in\mathfrak{g}'$, such that $\operatorname{Ca}_G\xi$ is a generic orbit with respect to our basis. For $\xi\in U$ we have $\delta_\lambda \xi\in U$ for each $\lambda\in\mathbb{R}^\times$ by equation \eqref{eq:dimension_generic_orbits_dilation}. Each orbit meets $U\cap\omega^\circ$ in exactly one point \cite[Theorem 3.1.9 and Theorem 4.5.5]{rep_nilpotent_lie_groups}. Furthermore, for any $\xi\in\omega^\times:=\omega^\circ\setminus\{0\}$, we have that
	\begin{equation*}
		\mathbb{R}^\times\to\omega^\times\colon \lambda \mapsto \delta_\lambda\xi
	\end{equation*}
	is a homeomoprhism. Thus also $U\cap\omega^\circ = \omega^\times=\{\delta_\lambda\ell\mid\lambda\in\mathbb{R}^\times\}$. But $\omega^\times$ also induces all maximal flat orbits, so they coincide with the generic orbits. Since the correspondence $\mathfrak{g}'/G \simeq \widehat G$ is a homeomorphism, we also have  $U/G \simeq \widehat G_\mathrm{gen}$ with respect to the subspace topologies. Let $q\colon U\to U/G$ be the quotient map. Now $q|_{\omega^\times}$ is a continuous bijection. We show, that it is also open. By \cite[Theorem 3.1.9]{rep_nilpotent_lie_groups}, there is a well define map $\psi\colon \omega^\times \times \mathfrak{z}^\circ \to U$, such that
	\begin{equation*}
		\psi(u,v) = w \quad \Leftrightarrow\quad w\in \operatorname{Ca}_G u\ \text{and} \ P_{\mathfrak{z}^\circ} w =v,
	\end{equation*}
	where $P_{\mathfrak{z}^\circ}$ is the projection onto $\mathfrak{z}^\circ$ along $\omega^\circ$. The map $\psi$ is a rational, non singular bijection with rational non singular inverse. Hence $\psi$ is a homeomorphism. If $V\subset \omega^\times$ is open in $\omega^\times$, then $\operatorname{Ca}_G V$ is open in $U$, since
	\begin{equation*}
		\psi(V\times \mathfrak{z}^\circ)= \operatorname{Ca}_G V.
	\end{equation*}
	Now, since $q$ is open and $q(\operatorname{Ca}_G V) = q(V)$, the restriction $q|_{\omega^\times}$ is an open map and thus a homeomorphism. If we now denote
	\begin{equation*}
		\sigma \colon \mathbb{R}^\times \to \widehat G_\mathrm{gen}\colon \lambda \mapsto [\delta_\lambda\pi],
	\end{equation*}
	then $\sigma$ is a homeomorphism by the discussion above. Let $\varphi \colon \widehat G \to [0,\infty)$ be Borel measurable. Then by Theorem 4.3.10 and the subsequent discussion in \cite{rep_nilpotent_lie_groups} 
	\begin{equation*}
		\int_{\widehat G}\varphi ([\pi])\,\,\mathrm{d}\widehat\mu([\pi]) = \int_{U\cap \omega^\circ}\varphi([\pi_\xi])|P\!\!f(\xi)|\,\,\mathrm{d}\widetilde \mu(\xi),
	\end{equation*}
	where $\widetilde \mu$ is the measure on $U\cap \omega^\circ$, such that $\{t\ell\mid t\in[0,1]\}$ has measure equal to one and $\pi_\xi \sim \operatorname{Ca}_G \xi$.
	Also, since our chosen Jordan-H\"older basis is an eigenbasis to $A$ resp.\ $\delta_\lambda$, we have
	\begin{equation*}
		|P\!\!f(\delta_\lambda \ell )| = |\det (\delta_\lambda\ell ([e_j,e_i]))_{j,i}|^{\frac{1}{2}} = |\det (|\lambda|^{n_i+n_j}\ell([e_j,e_i])  )_{j,i}|^{\frac{1}{2}} = |\lambda|^{\operatorname{Tr} B} |P\!\!f(\ell)|,
	\end{equation*}
	where $|\lambda|^{n_j}$ is the eigenvalue of $e_j$ to $\delta_\lambda$ for $j\in J$. Both $\sigma$ and $\sigma^{-1}$ are measurable and we have $\,\mathrm{d}(\widetilde \mu\circ\sigma)(\lambda) = \kappa |\lambda|^{\kappa-1}\,\,\mathrm{d}\lambda$. Hence
	\begin{equation*}
	\begin{split}
		\int_{U\cap \omega^\circ}\varphi([\pi_\xi])|P\!\!f(\xi)|\,\,\mathrm{d}\widetilde \mu(\xi) &= \int_{\mathbb{R}^\times}\varphi([\pi_\lambda])|P\!\!f(\delta_\lambda \ell)|\,\,\mathrm{d}(\widetilde \mu\circ\sigma)(\lambda) \\
		&= \int_{\mathbb{R}^\times }\varphi([\pi_\lambda])\kappa|\lambda|^{-1+\operatorname{Tr} A }|P\!\!f(\ell)|\,\,\mathrm{d}\lambda
	\end{split}
	\end{equation*}
	and $\sigma$ is a strict isomorphism of measure spaces.
\end{proof}

We will denote the euclidean Fourier transform on $\mathfrak{g}$ by
\begin{equation*}
	\mathcal F_\mathfrak{g}\varphi(\xi) = \int_\mathfrak{g} \mathrm{e}^{2\pi\mathrm{i} \xi(x)}\varphi(x)\,\mathrm{d}\mu(x),\quad \varphi\in\mathcal S(\mathfrak g),\ \xi\in\mathfrak{g}'.
\end{equation*}
Of course, there is exactly one Haar measure $\mu'$ on $\mathfrak{g}'$, such that the Fourier transform is a Gelfand triple isomorphism $\mathcal G(\mathfrak{g},\mu)\to\mathcal G(\mathfrak{g}',\mu')$.
Suppose $\ell\in\omega^\times$. The map
\begin{equation*}
	\wp_\ell f(\lambda,\xi):= f(\delta_\lambda (\ell+\xi))\quad \text{for}\quad \xi\in\mathfrak{z}^\circ, \lambda\in\mathbb{R}^\times \text{ and } f\colon \mathfrak{g}'\to\mathbb{C},
\end{equation*}
together with the euclidean Fourier transform and the Pedersen quantization will enable us to describe the group Fourier transform on $G$ (see also \cite{ruzhansky_flat_orbit_quantization} for a similar statement).

\subsection{The group Fourier transform on homogeneous groups}

Let us from now on always denote by $G$ a homogeneous Lie group with $\dim \mathfrak{z}=1$ and $\mathrm{SI/Z}(G)\neq \emptyset$. Trivially, the group Fourier transform is an isomorphism between $\mathcal S(G)$ and $\mathcal S(\widehat G)$. Also, the group Fourier transform is a unitary map from $L^2(G,\mu)$ to $L^2(\widehat G,\widehat \mu)$. Of course, we may define a Gelfand triple
\begin{equation*}
	\mathcal G(\widehat G,\widehat\mu) := (\mathcal S(\widehat G),L^2(\widehat G,\widehat \mu),\mathcal S'(\widehat G)),
\end{equation*}
such that $\mathcal F_G$ becomes a Gelfand triple isomorphism. Now we will use the isomorphism from Proposition \ref{prop:dilatation_dual_group_measure_isomorph} in order to find a new representation of the group Fourier transform on $L^2(G)$. This will be the basis for the definition of our new Gelfand triples and a Gelfand triple isomorphism in the form of an equivalent Fourier transform.

\begin{proposition}\label{prop:group_fourier_durch_dilatation}
	Suppose $\varphi\in \mathcal S(G)$ and $\pi\in\mathrm{SI/Z}_\mathbb{R}(G)$ with $\pi\sim\ell\in\omega^\times$, then
	\begin{equation*}
		\mathcal F_G \varphi (\pi_\lambda) = \begin{cases}\mathfrak{op}_\pi \big(\wp_\ell\mathcal F_{\mathfrak{g}}\varphi(\lambda,\cdot)\big), &\lambda>0,\\
		 \mathfrak{op}_{\overline{\pi}}\big(\wp_{\ell}\mathcal F_{\mathfrak{g}}\varphi(\lambda,\cdot)\big), &\lambda<0.
		\end{cases}
	\end{equation*}
\end{proposition}
\begin{proof}
	First of all, for any $\varphi\in \mathcal S(G)$, we have
	\begin{equation*}
		\mathcal F_G \varphi (\pi_\lambda) = \int_G \lambda^{-\operatorname{Tr} A}\varphi(\delta_\lambda^{-1} x)\pi(x)^*\,\,\mathrm{d}\mu(x) = \lambda^{-\operatorname{Tr} A}\mathcal F_G (\varphi\circ \delta_\lambda^{-1}) (\pi),
	\end{equation*}
	for $\lambda>0$. Also
	\begin{equation*}
		\mathcal F_\mathfrak{g} (\varphi\circ\delta_\lambda^{-1}) = \lambda^{\operatorname{Tr} A} (\mathcal F_\mathfrak{g} \varphi)\circ\delta_\lambda,
	\end{equation*}
	for $\lambda>0$. Notice, that for $x\in\mathfrak{g}$ and $z\in\mathfrak{z}$, we have $x\cdot z= x+ z$ and thus
	\begin{equation*}
		\mathrm{e}^{2\pi \mathrm{i} \ell(z)}\pi(x) = \pi(z)\pi(x) = \pi(z\cdot x) = \pi(z + x).
	\end{equation*}
	Let $\mu_\mathfrak{z}$ resp.\ $\nu$ be Haar measures on $\mathfrak{z}$ resp $\omega$, such that $\mu = \mu_\mathfrak{z}\otimes \nu$, then by the above calculation
	\begin{equation*}
	\begin{split}
		\mathcal F_G \varphi (\pi) &= \int_\omega \pi(x) \int_\mathfrak{z} \mathrm{e}^{-2\pi\mathrm{i} \ell(z)} \varphi(z-x)\,\,\mathrm{d} \mu_{\mathfrak{z}}(z)\,\mathrm{d}\nu(x) \\
		&= \int_\omega \pi(x) \int_{\mathfrak{z}^\circ} \mathrm{e}^{-2\pi\mathrm{i} \xi(X)} \mathcal F_{\mathfrak{g}}\varphi(\xi)\,\,\mathrm{d}\theta(\xi)\,\mathrm{d}\nu(x).
	\end{split}
	\end{equation*}
	Here $\theta$ is the measure associated to $\nu$ as described in Definition \ref{definition:pedersen_quant_4_flat_orbit}. This formula indeed holds pointwise. Hence
	\begin{equation*}
	\begin{split}
		\mathcal F_G \varphi (\pi_\lambda) &= \lambda^{-\operatorname{Tr} A}\mathfrak{op}_\pi( \mathcal F_\mathfrak{g}(\varphi \circ \delta_\lambda^{-1})) = \mathfrak{op}_\pi((\mathcal F_\mathfrak{g}\varphi)\circ\delta_\lambda )\\
		&= \mathfrak{op}_\pi\big(\wp_\ell\mathcal F_\mathfrak{g}\varphi(\lambda,\cdot)\big)
	\end{split}
	\end{equation*}
	for all $\lambda>0$. For $\lambda<0$ we get
	\begin{equation*}
		\mathcal F_G\varphi(\pi_\lambda) = \mathcal F_G(\overline\pi_{-\lambda}) = \mathfrak{op}_{\overline \pi} \big(\wp_{-\ell}\mathcal F_\mathfrak{g}\varphi(-\lambda,\cdot)\big),
	\end{equation*}
	since $\overline\pi\sim -\ell$. Now we can conclude the proof, by using $\delta_{-\lambda}(-\ell+\xi) = \delta_\lambda(\ell + \xi)$, for any $\xi\in \mathfrak{z}^\circ$.
\end{proof}

The above proposition (c.f. \cite[Theorem 3.3]{ruzhansky_flat_orbit_quantization}) shows that the group Fourier transform splits into operators which are easy to handle in the $L^2$-setting, if $\dim\mathfrak{z}=1$. If we use the isomorphism $(\widehat G,\widehat \mu)\simeq (\mathbb{R}^\times,\kappa|\lambda|^{Q-1}|P\!\!f(\ell)|\,\mathrm{d}\lambda)$ then we can see $\mathcal F_G$ as the composition of unitary operators
\begin{equation*}
\begin{split}
	\mathcal F_\mathfrak{g}&\colon L^2(G,\mu) \to L^2(\mathfrak{g}',\mu'),\\
	\wp_\ell&\colon L^2(\mathfrak{g}',\mu') \to L^2(\mathbb{R}^\times\times\mathfrak{z}^\circ;\kappa|\lambda|^{Q-1}|P\!\!f(\ell)|\,\mathrm{d}\lambda\,\mathrm{d}\theta(\xi)),\\
	\mathfrak{Op}_\pi &\colon L^2\big(\mathbb{R}^\times,\kappa|\lambda|^{Q-1}|P\!\!f(\ell)|\,\mathrm{d}\lambda;L^2(\mathfrak{z}^\circ,\theta)\big) \to L^2\big(\mathbb{R}^\times,\kappa|\lambda|^{Q-1}|P\!\!f(\ell)|\,\mathrm{d}\lambda; \mathcal{H\!S}(H_ \pi)\big),
\end{split}
\end{equation*}
where $\mathfrak{Op}_\pi = P_+\otimes \mathfrak{op}_\pi + P_-\otimes \mathfrak{op}_{\overline \pi}$, for the projection $P_\pm$ of $L^2(\mathbb{R}^\times)$ onto $L^2(\mathbb{R}^\pm)$. It is very convenient, that here the operator component emerges as a tensor product factor, which in turn enables us to understand multiplication operators more easily. This motivates us to define the following alternative spaces of test functions.

\subsection{The Fourier transform on $\mathcal S_*(G)$}

In order to know which function space is a good choice, we will first take a look at the pull back $\wp_\ell$. Here our earlier discussion of polynomial manifolds comes into play again. Remember that $\mathbb{R}^\times$ is equipped with a polynomial structure defined by $\mathbb{R}^\times = \mathbb{R}_+\, \dot\cup\,\mathbb{R}_-$, i.e. defined by the polynomial structures on $\mathbb{R}^\pm$. Similarly, we define $\mathfrak{g}_\ell^+$, $\mathfrak{g}^-_\ell$ and $\mathfrak{g}^\times$ by
\begin{equation*}
	\mathfrak{g}^\pm_\ell := \{t \ell + \eta \mid t\in\mathbb{R}_\pm, \eta\in\mathfrak{z}^\circ\} ,\quad \text{for }\ell\in\omega^\times = \omega^\circ\setminus\{0\}
\end{equation*}
and $\mathfrak{g}^\times = \mathfrak{g}^+_\ell\,\dot\cup\,\mathfrak{g}^-_\ell$ and equip $\mathfrak{g}^\pm_\ell$ with the polynomial structure analogously to the one on $\mathbb{R}^\pm$, i.e. the polynomial structure induced by the map
\begin{equation*}
	\mathfrak{g}^\pm_\ell\to \mathbb{R}\times\mathfrak{z}^\circ \colon (t\ell + \eta)\mapsto (t-1/t,\eta).
\end{equation*}
Then $\delta_\lambda$ induces a tempered diffeomorphism, as written in the following Lemma. The polynomial structure on $\mathfrak g^\times$ is just the one induced by its connected components. Notice that we just have $\mathfrak{g}^\times = \mathfrak{g}'\setminus \mathfrak{z}^\circ$ as a set.

\begin{lemma}\label{lemma:mathcalP_warmup}
	Let $\ell\in\omega^\times$. The Map $w_\ell\colon \mathbb{R}^\pm\times \mathfrak{z}^\circ\to\mathfrak{g}^\pm_\ell\colon (\lambda,\xi)\mapsto \delta_\lambda(\ell + \xi)$ is a tempered diffeomorphism.
\end{lemma}
\begin{proof}
	We prove that $\mathbb{R}^+\times\mathfrak{z}^\circ\simeq \mathfrak{g}^+_\ell$ via $w_\ell$. The proof to the second statement is analogous. Suppose $(\xi^j)_j$ is the dual basis to our Jordan--H\"older basis $(e_j)_{j=0}^{2n}$ of eigenvectors. Here $(\xi_j)_{j=1}^{2n}$ is the basis of $\mathfrak{z}^\circ$. Let $\kappa_j$ be the positive number, such that $\delta_\lambda \xi^j = \lambda^{\kappa_j}\xi^j$ for $\lambda>0$.
	
	We use the charts $\sigma$ resp.\ $\sigma_1$, defined by
	\begin{equation*}
		(\lambda, \sum_{j=1}^{2n} c_j\xi^j) \text{ resp.\ } (\lambda\ell + \sum_{j=1}^{2n} c_j\xi^j) \mapsto (\lambda - 1/\lambda, c_1,\dotsc,c_{2n}).
	\end{equation*}
	Then
	\begin{equation*}
	\begin{split}
		\sigma_1\circ w_\ell\circ\sigma^{-1} (t,c_1,\dotsc,c_{2n}) = &\left(\frac{(t + \sqrt{t^2 + 4})^{\kappa_0}}{2^{\kappa_0}}  - \frac{2^{\kappa_0}}{(t + \sqrt{t^2 + 4})^{\kappa_0}}, \right.\\
		& \ \ \left.\frac{(t + \sqrt{t^2 + 4})^{\kappa_1}}{2^{\kappa_1}}\,c_1,\dotsc,\frac{(t + \sqrt{t^2 + 4})^{\kappa_{2n}}}{2^{\kappa_{2n}}}\,c_{2n}\right),
	\end{split}
	\end{equation*}
	which is a slowly increasing function. Similarly
	\begin{equation*}
	\begin{split}
		\sigma\circ w_\ell^{-1}\circ\sigma^{-1}_1 (t,c_1,\dotsc,c_{2n})= &\left(\frac{(t + \sqrt{t^2 + 4})^\frac{1}{\kappa_0}}{2^\frac{1}{\kappa_0}}  - \frac{2^\frac{1}{\kappa_0}}{(t + \sqrt{t^2 + 4})^\frac{1}{\kappa_0}}, \right.\\
		& \ \ \left.\frac{(t + \sqrt{t^2 + 4})^\frac{-\kappa_1}{\kappa_0}}{2^\frac{-\kappa_1}{\kappa_0}}\,c_1,\dotsc,\frac{(t + \sqrt{t^2 + 4})^\frac{-\kappa_{2n}}{\kappa_{0}}}{2^\frac{-\kappa_{2n}}{\kappa_{0}}}\,c_{2n}\right)
	\end{split}
	\end{equation*}
	is slowly increasing. 
\end{proof}

By Lemma \ref{lemma:reducedSchwartzspace_carries_SubTop}, we can see $\mathcal S(\mathfrak{g}^\pm_\ell)$ as the space
\begin{equation*}
	\mathcal S(\mathfrak{g}^\pm_\ell) = \{\varphi \in\mathcal S(\mathfrak{g}')\mid \varphi\equiv 0 \text{ on }\mathfrak{g}^\mp_\ell\},
\end{equation*}
equipped with the subspace topology in $\mathcal S(\mathfrak{g}')$.

The tempered diffeomorphism from the last lemma induces a Gelfand triple isomorphism.

\begin{lemma}\label{lemma:mathcalP}
	The pullback $\wp_\ell f:= f\circ w_\ell$ defines a Gelfand triple isomorphism
	\begin{equation*}
		\wp_\ell \colon \mathcal G(\mathfrak{g}_\ell^\pm,\mu')  \to \mathcal G(\mathbb{R}^\pm,\kappa |P\!\!f(\ell)|\,|\lambda|^{Q-1}\,\mathrm{d}\lambda)
		\otimes\mathcal G(\mathfrak{z}^\circ,\theta),
	\end{equation*}
	where $\kappa$, $|P\!\!f(\ell)|$ are the constants introduced in Proposition \ref{prop:dilatation_dual_group_measure_isomorph} and the preceding remarks and $Q$ is the homogenous dimension of $G$.
\end{lemma}
\begin{proof}
	We take an arbitrary $f\in C_c(\mathfrak{g}_\ell^\pm)$. Define $\omega^\pm := \mathbb{R}^\pm\cdot\ell$, then
	\begin{equation*}
	\begin{split}
		\int_{\mathbb{R}^\pm}\int_{\mathfrak{z}^\circ} f(\delta_\lambda(\ell+\xi))\kappa |P\!\!f(\ell)|\,&|\lambda|^{Q-1}\,\mathrm{d} \lambda \,\mathrm{d} \theta(\xi)\\ &= \int_{\mathbb{R}^\pm}\int_{\mathfrak{z}^\circ} f((\delta_\lambda\ell) + \xi) \kappa |P\!\!f(\ell)|\,|\lambda|^{\kappa_0 - 1}\,d\lambda \,\mathrm{d} \theta(\xi)\\
		&= \int_{\omega^\pm}\int_{\mathfrak{z}^\circ} f(\eta + \xi)) \, |P\!\!f(\ell)|\,d\mu_{\omega^\circ}(\eta) \,\mathrm{d} \theta(\xi)\\
		&= \int_{\mathfrak{g}_\ell^\pm} f(\xi)\,\,\mathrm{d}\mu'(\xi).
		\end{split}
	\end{equation*}
	For the last two lines we used that the measure $\mu_{\omega^\circ}$ on $\omega^\circ$ is defined by the Lebesgue measure and $\ell$ and that $\theta$ is defined by $\mu' = |P\!\!f(\ell)|\,\mu_{\omega^\circ}\otimes \theta $. The rest follows with the fact, that $\wp_\ell f = f\circ w_\ell$, where $w_\ell$ is the tempered diffeomorphism from Lemma \ref{lemma:mathcalP_warmup}.
\end{proof}

We also proved that the restriction of the Haar measure $\mu'$ to $\mathfrak g^\pm_\ell$ is actually a tempered measure with respect to our chosen polynomial structure. 

Now we are ready to define Gelfand triples, with respect to which we get a convenient theory for the group Fourier transform.

\begin{definition}
	We define the following reduced Schwartz space
	\begin{equation*}
		\mathcal S_*(G) := \{ \varphi\in\mathcal S(G)\mid [(\lambda,x)\mapsto \varphi(\lambda z+x)]\in\mathcal S_*(\mathbb{R})\,\hat\otimes\,\mathcal S(\omega)\}
	\end{equation*}
	for any choice $z\in \mathfrak{z}\setminus\{0\}$, equipped with the subspace topology in $\mathcal S(G)$, and the corresponding Gelfand triple
	\begin{equation*}
		\mathcal G_*(G,\mu) := (\mathcal S_*(G), L^2(G,\mu),\mathcal S'_*(G)),
	\end{equation*}
	equipped with the real structure given by the pointwise complex conjugation. Furthermore, we define the Gelfand triple
\begin{equation*}
	\mathcal G (\mathbb{R}^\times;\pi) := \left(\begin{matrix} \mathcal S(\mathbb{R}^\times;\pi)\\ L^2(\mathbb{R}^\times;\pi) \\ \mathcal S'(\mathbb{R}^\times;\pi) \end{matrix}\right):= \mathcal G(\mathbb{R}^\times, \kappa\,|P\!\! f(\ell)|\,|\lambda|^{Q-1}\,\mathrm{d}\lambda)\otimes \mathcal{G}_\mathrm{op}(\pi).
\end{equation*}
for each $\mathrm{SI/Z}_\mathbb{R}(G)\ni\pi\sim\ell \in\omega^\circ$.
\end{definition}
That $\mathcal G_*(G)$ is indeed a Gelfand triple can be seen by using Proposition \ref{prop:distr_conv_to_zero_embedding}. We use any linear isomorphism $\mathbb{R}\simeq \mathfrak{z}$ to define $\dot{\mathcal B}_*'(\mathfrak{z};\mathcal S'(\omega))$, then we see, since $L^2(G,\mu)\subset \dot{\mathcal B}_*'(\mathfrak{z};\mathcal S'(\omega))$ by Lemma \ref{lemma:lebesgue_bochner}, that the space $L^2(G,\mu)$ is embedded into $\mathcal S_*'(G) = \mathcal S'_*(\mathfrak{z};\mathcal{S}'(\omega))$. This embedding is continuous, since the embedding $L^2(G,\mu)\hookrightarrow \mathcal S'(G)$ is continuous. Of course, the canonical map of $\mathcal S_*(G)$ into $L^2(G,\mu)$ is a continuous embedding as well. Now the Hahn--Banach theorem implies that both embeddings are also dense, for they are dual to each other.

To be more precise, if $\mathcal S_*(G)^\perp$ is the polar of $\mathcal S_*(G)$ in $L^2(G,\mu)$, then it is also the kernel of the dual map $L^2(G,\mu)\to \mathcal S_*'(G)$. But this map has a trivial kernel by Lemma \ref{lemma:lebesgue_bochner}. Hence $\mathcal S_*(G)^\perp=\{0\}$ and $\mathcal S_*(G)$ is dense in $L^2(G,\mu)$. Now denote by $Y$ the image of $L^2(G,\mu)$ in $\mathcal S_*'(G)$. Since $\mathcal S_*(G)$ is reflexive, $Y^\circ$ can be identified with the kernel of the embedding $\mathcal S_*(G)\hookrightarrow L^2(G,\mu)$, which is trivial. Hence $Y\subset \mathcal S_*'(G)$ is dense as well.

Notice that $\mathcal G_*(G,\mu)$ does not depend either on the choice of some $\pi\in\mathrm{SI/Z}_\mathbb{R}(G)$ or some $z\in\mathfrak{z}$. The Gelfand triple $\mathcal G(\mathbb R^\times;\pi)$ does depend on $\pi\in\mathrm{SI/Z}_\mathbb{R}(G)$ but each different choice of $\pi$ leads to an isomorphic Gelfand triple as the theorem below shows.

\begin{theorem}
	Let $\mathrm{SI/Z}_\mathbb{R}(G)\ni\pi\sim \ell\in\omega^\times$. Let the \emph{Fourier transform in $\pi$-picture}, $\mathcal F_\pi$ be defined by
	\begin{equation*}
		\mathcal F_\pi := \mathfrak{Op}_\pi \circ\wp_\ell\circ\mathcal F_{\mathfrak{g}},
	\end{equation*}
	where $\mathfrak{Op}_\pi= P_+\otimes \mathfrak{op}_\pi + P_-\otimes \mathfrak{op}_{\overline \pi}$ and $P_+=1 -P_-$ is the projection of $\mathcal S(\mathbb{R}^\times)$ onto $\mathcal S(\mathbb{R}^+)$ along $\mathcal S(\mathbb{R}^-)$. Then $\mathcal F_\pi$ is a Gelfand triple isomorphism
	\begin{equation*}
		\mathcal F_\pi : \mathcal G_*(G) \to \mathcal G(\mathbb{R}^\times;\pi).
	\end{equation*}
\end{theorem}
\begin{proof}
	The proof essentially writes itself by now and is a summary of previous statements. The euclidean Fourier transform $\mathcal F_\mathfrak{g}$ is a Gelfand triple isomorphism between $\mathcal G_*(\mathfrak g,\mu)$ and $\mathcal G(\mathfrak{g},\mu')=\mathcal G(\omega^\times,|P\!\!f(\ell)|\,\mu_{\omega^\circ})\otimes \mathcal G(\mathfrak{z}^\circ,\theta)$ by Lemma \ref{lemma:fouriertransform_on_homogenuous_Schartz_spaces}, where we choose the Haar measures $\mu_{\mathfrak{\omega^\circ}}$ and $\theta$, such that $\mu' = |P\!\!f(\ell)|\,\mu_{\mathfrak{\omega^\circ}}\otimes \theta$ and $\mu_{\mathfrak{\omega^\circ}}$ is induced by the Lebesgue measure $\,\mathrm{d}\lambda$ via the map $\mathbb{R}\ni\lambda\mapsto \lambda\ell\in\omega^\circ$.
	
	By Lemma \ref{lemma:mathcalP}, the pull back $\wp_\ell$ is a Gelfand triple isomorphism between $\mathcal G(\mathfrak g^\times,\mu_{\mathfrak{g}'})$ and $\mathcal G(\mathbb{R}^\times,\kappa\,|P\!\!f(\ell)|\,|\lambda|^{Q-1}\,\mathrm{d}\lambda)\otimes\mathcal G(\mathfrak{z}^\circ,\theta)$.
	
	For the last step we just need to use that $\mathfrak{Op}_\pi= P_+\otimes\mathfrak{op}_\pi + P_-\otimes\mathfrak{op}_{\overline \pi}$ is a Gelfand triple isomorphism between $\mathcal G(\mathbb{R}^\times,\kappa_0\,|P\!\!f(\ell)|\,\,\mathrm{d}\lambda)\otimes\mathcal G(\mathfrak{z}^\circ,\mu_{\mathfrak{z}^\circ})$ and $\mathcal G(\mathbb{R}^\times;\pi)$ by Theorem \ref{lemma:orbital_pedersen_quant_gelfand_triple} and Definition \ref{definition:pedersen_quant_4_flat_orbit}.
\end{proof}

Let us now discuss a few properties of $\mathcal S_*(G)$ and $\mathcal S(\mathbb{R}^\times;\pi)$. Their duals can be identified with quotient spaces, in particular
\begin{equation*}
\begin{split}
	\mathcal S_*'(G)&\simeq \mathcal S'(G)/(\mathcal P(\mathfrak{z})\otimes \mathcal S'(\omega))\quad\text{and}\\
	\mathcal S'(\mathbb{R}^\times;\pi)&\simeq\mathcal S(\mathbb{R}^\times)\,\hat\otimes\, \mathcal L(H^{-\infty}_\pi,H^\infty_\pi)/(\mathcal E'_0(\mathbb R)\otimes \mathcal L(H^{-\infty}_\pi,H^\infty_\pi)),
\end{split}
\end{equation*}
by Lemma \ref{lemma:char_dual_S0} and Corollary \ref{corollary:char_dual_Sstar}. By employing Proposition \ref{prop:distr_conv_to_zero_embedding}, we can identify a large space of distributions on $G$ resp.\ $\mathbb{R}$ that are embedded into $\mathcal S_*'(G)$ resp.\ $\mathcal S'(\mathbb{R}^\times;\pi)$. I.e. if we define $\dot{\mathcal B}'(\mathfrak{z};\mathcal S'(\omega))$ by using any isomorphism $\mathbb{R}\simeq \mathfrak{z}$, then
\begin{equation*}
	\dot{\mathcal B}'(\mathfrak{z};\mathcal S'(\omega)) \hookrightarrow \mathcal S'_*(G)\quad \text{and}\quad \widetilde{\mathcal B}'(\mathbb{R};\mathcal L(H^\infty_\pi,H^{-\infty}_\pi))\hookrightarrow \mathcal S'(\mathbb{R}^\times;\pi).
\end{equation*}
We may, for example, identify $L^p(G,\mu)$, for $p\in [1,\infty)$, and also $\mathcal S(G)$ as a subspaces of $\dot{\mathcal B}'(\mathfrak{z};\mathcal S'(\omega))$ and the Bochner-Lebesgue spaces $L^p(\mathbb{R},\,\mathrm{d}\lambda;\mathcal L(H_\pi))$, for $p\in (1,\infty]$, and also $\mathcal S(\mathbb{R};\mathcal L(H^{-\infty}_\pi,H^\infty_\pi))$ as subspaces of $\widetilde{\mathcal B}'(\mathbb{R};\mathcal L(H^\infty_\pi,H^{-\infty}_\pi))$.

The definition of $\mathcal S(\mathbb{R}^\times;\pi)$ and $\mathcal S'(\mathbb{R}^\times;\pi)$ enables us to define a multiplication with a large class of smooth functions via Theorem \ref{theorem:allg_fortsetzung_bilinear}.

\begin{proposition}
	For any $\pi\in\mathrm{SI/Z}(G)$, the multiplications
	\begin{equation*}
	\begin{split}
		\mathcal O_\mathrm{M}(\mathbb{R}^\times; \mathcal L(H^\infty_\pi))\times \mathcal S(\mathbb{R}^\times;\pi) &\colon (f,\varphi)\mapsto f\,\varphi \\
		\mathcal O_\mathrm{M}(\mathbb{R}^\times; \mathcal L(H^{-\infty}_\pi))\times \mathcal S(\mathbb{R}^\times;\pi) &\colon (f,\varphi)\mapsto \varphi\,f,
	\end{split}
	\end{equation*}
	defined pointwise by composition of operators in $\mathcal L(H^{-\infty}_\pi,H^\infty_\pi)$, $\mathcal L(H^\infty_\pi)$ and $\mathcal L(H^{-\infty}_\pi)$, are hypocontinuous bilinear maps.
\end{proposition}
\begin{proof}
	We just need show that we may apply Theorem \ref{theorem:allg_fortsetzung_bilinear}. The compositions of operators
	\begin{equation*}
	\begin{split}
		\mathcal L(H^\infty_\pi)\times \mathcal L(H^{-\infty}_\pi,H^\infty_\pi) &\colon (A,B)\mapsto A\,B \\
		\mathcal L(H^{-\infty}_\pi)\times \mathcal L(H^{-\infty}_\pi,H^\infty_\pi) &\colon (A,B)\mapsto B\,A
	\end{split}
	\end{equation*}
	are hypocontinuous, since separately continuous maps on barrelled spaces are hypocontinuous. Since $H^\infty_\pi\simeq \mathcal S(\mathbb{R}^n)$ and Theorem \ref{theorem:allg_fortsetzung_bilinear}, the above operator spaces are barrelled by \cite[Corollaire 2 on page 128 of chapter 2]{grothendieck_tensorprodukte} and \cite[Corollary to 8.4 in chapter 2]{schaefer_tvs}. Also, the multiplication of slowly increasing functions and Schwartz functions is hypocontinuous. This follows directly form the definition and comments on pages 243 and 244 of \cite{Schwartz_distributions}. Now we just need to remind ourselves, that $\mathcal S(\mathbb{R}^\times;\pi)$ is a tensor product of nuclear Fr\'echet spaces.
\end{proof}

Now, we will prove the analogous result for the multiplication with the operator valued tempered distributions $\mathcal S'(\mathbb{R}^\times;\pi)$. As we used in the proof above, $H^\infty_\pi$ is reflexive for any $\pi\in\mathrm{SI/Z}(G)$. Thus, by using the transpose, we get the two isomorphisms of topological vector spaces
\begin{equation*}
	\mathcal L(H^\infty_\pi)\ni A\mapsto A^\mathrm{t}\in \mathcal L(H^{-\infty}_\pi)\quad\text{and}\quad \mathcal L(H^{-\infty}_\pi)\ni B\mapsto B^\mathrm{t}\in \mathcal L(H^{\infty}_\pi).
\end{equation*}
Denote for $f $ in $\mathcal O_{\rm M}(\mathbb{R}^\times;\mathcal L(H^\infty_\pi))$ or in $\mathcal O_{\rm M}(\mathbb{R}^\times;\mathcal L(H^{-\infty}_\pi))$ the operator valued function $f^\mathrm{t}(\lambda):=f(\lambda)^\mathrm{t}$. Then we may define multiplications on $\mathcal S'(\mathbb{R}^\times;\pi)$ by
\begin{equation*}
	(f\,\phi)(\varphi):=\phi(f^\mathrm{t}\,\varphi)\quad \text{and}\quad (\phi\,g)(\varphi):=\phi(\varphi\,g^\mathrm{t}),
\end{equation*}
for all $\phi\in\mathcal S'(\mathbb{R}^\times;\pi)$ and $\varphi\in\mathcal S(\mathbb{R}^\times;\pi)$, if we choose $f\in \mathcal O_{\rm M}(\mathbb{R}^\times;\mathcal L(H^{-\infty}_\pi))$ and $g\in\mathcal O_{\rm M}(\mathbb{R}^\times;\mathcal L(H^\infty_\pi))$. We get the following corollary.
\begin{corollary}
	For any $\pi\in\mathrm{SI/Z}(G)$, the multiplications
	\begin{equation*}
	\begin{split}
		\mathcal O_\mathrm{M}(\mathbb{R}^\times; \mathcal L(H^\infty_\pi))\times \mathcal S'(\mathbb{R}^\times;\pi) &\colon (f,\phi)\mapsto f\,\phi, \\
		\mathcal O_\mathrm{M}(\mathbb{R}^\times; \mathcal L(H^{-\infty}_\pi))\times \mathcal S'(\mathbb{R}^\times;\pi) &\colon (f,\phi)\mapsto \phi\,f
	\end{split}
	\end{equation*}
	are hypocontinuous.
\end{corollary}
\begin{proof}
	This follows directly from the definition of the multiplication and the fact that the dual pairing is hypocontinuous. Equivalently, we could also directly employ Theorem \ref{theorem:allg_fortsetzung_bilinear}.
\end{proof}

Let us now relate the Fourier transform in $\pi$ picture with the group Fourier transform. Denote by $j_\pi$ the map $j_\pi(\sigma):= [\mathbb{R}^\times\ni \lambda\mapsto \sigma(\pi_\lambda)\in\mathcal L(\mathcal H^{-\infty}_\pi,\mathcal H^\infty_\pi)]$, defined on $\mathcal S(\widehat G)$.

Due to Proposition \ref{prop:dilatation_dual_group_measure_isomorph}, we know that $j_\pi$ has a unitary extension from $L^2(\widehat G_\mathrm{gen},\widehat \mu)$ onto $L^2(\mathbb{R}^\times;\pi)$. Because the Plancherel measure $\widehat\mu$ is concentrated on $\widehat G_\mathrm{gen}$ \cite[Theorem 4.3.16]{rep_nilpotent_lie_groups}, we can see $j_\pi$ as a map defined on $L^2(\widehat G,\widehat \mu)$. Proposition \ref{prop:group_fourier_durch_dilatation} implies the $L^2$-diagram below.

\begin{center}
\begin{small}
\begin{tikzcd}
	L^2(G,\mu) \arrow[r,"\mathcal F_G","\simeq"'] \arrow[rd,"\mathcal F_\pi"',"\simeq"]
		& L^2(\widehat G,\widehat \mu) \arrow[d,"j_\pi","\simeq"'] \!\!
		& \mathcal S(G) \arrow[r,"\mathcal F_G","\simeq"']
		& \mathcal S(\widehat G) \!\!
		& \mathcal S'(G) \arrow[r, "\simeq"',"\mathcal F_G"] \arrow[d, "j_*'"']
		& \mathcal S'(\widehat G) \arrow[d, "j_0'"]\\
		& L^2(\mathbb R^\times;\pi)		\!\!
		& \mathcal S_*(G) \arrow[u, "j_*", "\subset"',hook] \arrow[r, "\mathcal F_G", "\simeq"']
		& \mathcal S(\widehat G_\mathrm{gen}) \!\! \arrow[u, "j_0"', "\subset",hook]
		& \mathcal S_*'(G) \arrow[r,"\simeq"', "\mathcal F_G"] \arrow[rd,"\simeq", "\mathcal F_\pi"']
		& \mathcal S'(\widehat G_\mathrm{gen}) \arrow[d,"\simeq"', "(j_\pi')^{-1}"]\\
		& %
		& %
		& \mathcal S(\mathbb R^\times;\pi) \!\! \arrow[ul,"\mathcal F_\pi^{-1}","\simeq"'] \arrow[u,"j_\pi^{-1}"',"\simeq"]
		& %
		& \mathcal S'(\mathbb R^\times;\pi)\\
\end{tikzcd}
\end{small}
\label{diagram:fourier_transforms}
\end{center}

For the Schwartz spaces and spaces of tempered distributions we get a very similar diagram. Denote also by $\mathcal S(\widehat G_\mathrm{gen})$ the image of $\mathcal S(\mathbb{R}^\times;\pi)$ under $j_\pi^{-1}$. The commutative diagram for the $L^2$-spaces implies the commutative diagram for the Schwartz spaces above. Then, by duality, we get the commutative diagram for the tempered distributions. However, the group Fourier transformations $\mathcal F_G$ are defined by duality on $\mathcal S'(G)$ resp.\ $\mathcal S_*'(G)$ and are not the same map, even though we use the same symbol.

By Corollary \ref{corollary:char_dual_Sstar} the map $j'_*$ can be seen as the quotient map
\begin{equation*}
	\mathcal S'(G) \to \mathcal S'(G)/(\mathcal{P}(\mathfrak{z})\otimes \mathcal S'(\omega))\simeq \mathcal S'_*(G),
\end{equation*}
which is an open map. This also implies that $j'_0$ is surjective and open.

Since $w_\ell\colon \mathbb{R}^\times \times \mathfrak{z}^\circ \to \mathfrak{g}^\times$, $w_\ell(\lambda,\xi) = \delta_\lambda(\ell +\xi)$ is a tempered diffeomorphism, we can also see $\wp_\ell$ as an isomorphism between $\mathcal O_{\rm M}(\mathfrak{g}^\times)$ and $\mathcal O_{\rm M}(\mathbb{R}^\times \times\mathfrak{z}^\circ)$ resp.\ between $\mathcal S(\mathfrak{g}^\times)$ and $\mathcal S(\mathbb{R}^\times\times\mathfrak{z}^\circ)$. However, in order to examine the Fourier image of $\mathcal S(G)$, it is even better to consider mixed spaces. We equip $\omega^\times = \mathbb{R}^\times \cdot \ell$ with the polynomial structure transported from $\mathbb{R}^\times$. The space $\mathcal O_\mathrm{M}(\omega^\times)\,\hat\otimes\,\mathcal S(\mathfrak{z}^\circ)$ can be seen as a subspace of $\mathcal O_\mathrm{M}(\mathfrak{g}^\times)$. In this manner we define $\wp_\ell$ on $\mathcal O_\mathrm{M}(\omega^\times)\,\hat\otimes\,\mathcal S(\mathfrak{z}^\circ)$.

\begin{lemma}
	The Gelfand-Triple isomorphism $\wp_\ell$ restricts to an isomorphism
	\begin{equation*}
		\wp_\ell\colon \mathcal O_{\rm M}(\omega^\times)\,\hat\otimes\,\mathcal S(\mathfrak{z}^\circ)\to \mathcal O_{\rm M}(\mathbb{R}^\times)\,\hat\otimes\,\mathcal S(\mathfrak z^\circ).
	\end{equation*}
\end{lemma}
\begin{proof}
	We identify $\omega^\times\simeq \mathbb{R}^\times$ and $\mathfrak{z}^\circ\simeq \mathbb{R}^{2n}$ and $\mathfrak{z}^\circ\simeq \mathbb{R}^{2n}$ via our basis of eigenvectors to the dilations. As usual, it is enough to consider the $\mathbb{R}^+$-part, since $\mathcal O_\mathrm{M}(\mathbb{R}^\times) = \mathcal O_\mathrm{M}(\mathbb{R}^+)\oplus \mathcal O_\mathrm{M}(\mathbb{R}^-)$. With these adjustments, we need to exchange $\wp_\ell$ by the map $\wp$, where
	\begin{equation*}
		\wp g(\lambda,x) = g(\lambda^{\kappa_0}, (\lambda^{\kappa_j}x_j)_{j=1}^{2n}).
	\end{equation*}
	First of all, we realize that $\lambda\mapsto \lambda^{\kappa_0}$ is a tempered diffeomorphism. Hence $T\in\mathcal L(\mathcal O_{\rm M}(\mathbb{R}_+))$, where $T\psi(\lambda):=\psi(\lambda^{\kappa_0})$, is an isomorphism.
	
	Now let us define linear isomorphisms $f_\lambda(x) = (\lambda^{\kappa_j / \kappa_0}x_j)_{j=1}^{2n}$ on $\mathbb{R}^{2n}$. Then it is easy to see, that both $\lambda\mapsto f_\lambda$ and $\lambda\mapsto f_\lambda^{-1}$ define functions in $\mathcal{O}_{\rm M}(\mathbb{R}^+;\mathcal L(\mathbb{R}^{2n}))$ with values in $\mathrm{Gl}(\mathbb{R}^{2n})$. Now we denote by $F_\lambda$ the corresponding operator $F_\lambda \varphi := \varphi\circ f_\lambda$ and set $F\colon \lambda\mapsto F_\lambda$ resp.\ $F^{-1}\colon \lambda\mapsto F^{-1}_\lambda$. A standard calculation shows that for any continuous seminorm $p$ on $\mathcal L(\mathcal S(\mathbb{R}^{2n}))$ and any $k\in\mathbb{N}_0$, there is a polynomial $q$ on $\mathcal L(\mathbb{R}^{2n})^{k+2}$, such that
	\begin{equation*}
		p(\partial_\lambda^k F_\lambda) \le q(f_\lambda^{-1},f_\lambda,\partial_\lambda f_\lambda,\dotsc,\partial^k_\lambda f_\lambda).
	\end{equation*}
	Of course, an analogous inequality is valid for $F^{-1}$. Hence, we may conclude
	\begin{equation*}
		F,F^{-1}\in\mathcal{O}_{\rm M}(\mathbb{R}^+;\mathcal L(\mathcal S(\mathbb{R}^{2n}))).
	\end{equation*}
	Here, $F^{-1}$ is indeed the inverse of $F$ in the algebra $\mathcal{O}_{\rm M}(\mathbb{R}^+;\mathcal L(\mathcal S(\mathbb{R}^{2n})))$. Due to Theorem \ref{theorem:allg_fortsetzung_bilinear}, we know that the multiplication
	\begin{equation*}
		\mathcal O_{\rm M}(\mathbb{R}^+;\mathcal S(\mathbb{R}^{2n}))\ni g\mapsto F\, g\in \mathcal O_{\rm M}(\mathbb{R}^+;\mathcal S(\mathbb{R}^{2n})),\quad (F\, g)(\lambda,x ) = H_\lambda(g(\lambda,\cdot))(x),
	\end{equation*}
	is continuous and in fact an isomorphism.
	
	Because $\wp g = (T\otimes 1)(F\, g)$, we can conclude that $\wp$ is an isomorphism.
\end{proof}

Using the above lemma, we may now prove the following continuity property for the Fourier transform in $\pi$-picture on $\mathcal S(G)$.

\begin{proposition}\label{proposition:schwartz_fourier_picture}
	The Fourier transform in $\pi$-picture restricts to a continuous map
	\begin{equation*}
		\mathcal F_\pi\colon \mathcal S(G) \to \mathcal O_{\rm M}(\mathbb{R}^\times)\,\hat\otimes\, \mathcal L(\mathcal H_\pi^{-\infty},\mathcal H_\pi^\infty).
	\end{equation*}
\end{proposition}
\begin{proof}

	This statement follows from the continuity of the maps
	\begin{equation*}
		\mathcal S(G) \xrightarrow{\mathcal F_\mathfrak{g}} \mathcal S(\mathfrak{g}') \hookrightarrow \mathcal O_{\rm M}(\omega^\times)\,\hat\otimes\, \mathcal S(\mathfrak{z}^\circ)\xrightarrow{\wp_\ell}\mathcal O_{\rm M}(\mathbb{R}^\times)\,\hat\otimes\, \mathcal S(\mathfrak{z}^\circ),
	\end{equation*}
	in which we use the continuous inclusion $\mathcal S(\omega^\circ)\subset\mathcal O_{\rm M}(\omega^\times)$, and also from the continuity of
	\begin{equation*}
		\mathfrak{Op}_\pi = P_+\otimes\mathfrak{op}_\pi + P_-\otimes\mathfrak{op}_{\overline \pi} \colon \mathcal O_{\rm M}(\mathbb{R}^\times)\,\hat\otimes\, \mathcal S(\mathfrak{z}^\circ)\to \mathcal O_{\rm M}(\mathbb{R}^\times)\,\hat\otimes\, \mathcal L(\mathcal H_\pi^{-\infty},\mathcal H_\pi^\infty).
	\end{equation*}
\end{proof}

\section{Gelfand triples for the Kohn-Nirenberg quantization}

In \cite{ruzhansky_nilpotent} a pseudo-differential calculus resp.\ a Kohn-Nirenberg quantization for graded nilpotent Lie groups was developed. We will embed this definition into our context and derive an integral formulation for the Kohn-Nirenberg quantization for a general class of symbols. First, consider the map
\begin{equation*}
	\mathcal T\colon \mathcal S(G\times G) \to \mathcal S(G\times G),\ \mathcal Tf(x,y):= f(x,xy^{-1}).
\end{equation*}
Then, it is easy to see that $\mathcal T$ extends to a Gelfand triple isomorphism
\begin{equation*}
	\mathcal T\colon \mathcal G(G,\mu)\otimes \mathcal G(G,\mu)\to \mathcal G(G,\mu)\otimes \mathcal G(G,\mu).
\end{equation*}
Denote by $\mathcal K$ the kernel map
\begin{equation*}
	\mathcal K\colon \mathcal L(\mathcal G(G,\mu),\mathcal G(G,\mu))\to \mathcal G(G,\mu)\otimes \mathcal G(G,\mu)
\end{equation*}
from Lemma \ref{lemma:allg_kernsatz_gelfand_triple}. We may define the Kohn-Nirenberg quantization as the Gelfand triple isomorphism
\begin{equation*}
	\operatorname{Op} := \mathcal K^{-1}\mathcal T^{-1}(1\otimes\mathcal F_G^{-1})\colon \mathcal G(G,\mu)\otimes\mathcal G(\widehat G,\widehat \mu)\to \mathcal L(\mathcal G(G,\mu),\mathcal G(G,\mu)).
\end{equation*}
That means for $a\in L^2(G\times G,\mu\otimes\mu)$, we have $\operatorname{Op}(a)\in\mathcal{H\! S}(L^2(G))$ and
\begin{equation*}
	\boldsymbol{(} \operatorname{Op}(a)f,g\boldsymbol{)}_{L^2(G,\mu)} =  \int_{\widehat G}\int_G \operatorname{Tr} [a(x,\pi)\, ((1\otimes \mathcal F_G\operatorname{inv})\mathcal T g\otimes\overline{f})(x,\pi)^*]\,\mathrm{d}\mu(x)\,\mathrm{d}\widehat\mu([\pi])
\end{equation*}
for all $f,g\in L^2(G,\mu)$, where $\operatorname{inv} f(x) := f(-x)$ and $\boldsymbol{(} \cdot,\cdot\boldsymbol{)}_{L^2(G,\mu)}$ is the inner product in $L^2(G,\mu)$. We denote the right translation of functions $f$ by $R_x f(y):= f(yx)$. Because 
\begin{align*}
	\operatorname{Tr} [a(x,\pi)\, ((1\otimes \mathcal F_G\operatorname{inv})\mathcal T g\otimes\overline{f})(x,\pi)^* ] &= \operatorname{Tr} \big[a(x,\pi)\, \big( g(x)\,(\mathcal F_G R_x^{-1}\operatorname{inv} \overline f)(\pi)\big)^*\big]\\
	&= \overline{g(x)}\operatorname{Tr} \big[a(x,\pi)\, (\mathcal F_G\operatorname{inv} \overline f)(\pi)^*\,\pi(x)\big]\\
	&= \overline{g(x)}\, \operatorname{Tr}[ a(x,\pi)\,\mathcal F_G f(\pi)\,\pi(x)]
\end{align*}
for almost all $(x,[\pi])\in G\times\widehat G$, we may write the operator $\operatorname{Op}(a)$ as
\begin{equation*}
	\operatorname{Op}(a)\varphi = \int_{\widehat G} \operatorname{Tr} [\pi(\cdot)\,a(\cdot,\pi)\,\mathcal F_G f(\pi) ]\,\mathrm{d}\widehat \mu([\pi]),\quad \text{for }f\in L^2(G),
\end{equation*}
where the integral converges in $L^2(G,\mu)$.

We will now define the Kohn-Nirenberg quantization in the context of the Gelfand triples $\mathcal G_*(G.\mu)$ and $\mathcal G(\mathbb{R}^\times;\pi)$ for $\pi\in\mathrm{SI/Z}_\mathbb{R}(G)$. Subsequently, we will discuss an integral formula similar to the $L^2$-case above, but for a different class of symbols.

\subsection{The Kohn-Nirenberg quantization for operators defined on $\mathcal S_*(G)$}

We still take $G$ to be a homogeneous Lie group with $\dim\mathfrak{z}=1$ and $\pi\in\mathrm{SI/Z}_\mathbb{R}(G)$. We already saw that $\mathcal F_G$ is a Gelfand triple isomorphism from $\mathcal G(G)$ to $\mathcal G(\widehat G)$, by the definition of $\mathcal G(\widehat G)$. Now we want to use the corresponding statement for the $\mathcal S_*(G)$ test functions. Again we need to show that $\mathcal T$ is a Gelfand triple isomorphism in this context. Although the map $\mathcal T$ is not well defined on $\mathcal G_*(G)\otimes\mathcal G_*(G)$, it is well defined on $\mathcal G(G)\otimes \mathcal G_*(G)$.

\begin{lemma}
	The map $\mathcal T\restriction_{\mathcal S(G\times G)}$ extends to a Gelfand triple from $\mathcal G(G)\otimes\mathcal G_*(G)$ onto itself, which we will also call $\mathcal T$ by a slight abuse of notation.
\end{lemma}
\begin{proof}
	Suppose $\varphi \in \mathcal S(G)\,\hat\otimes\,\mathcal S_*(G)$ and $q\in\mathcal P(G)$, then for all $x\in G$ and $y\in\omega$
	\begin{equation*}
		\int_{\mathfrak{z}} q(z)\varphi (x,x(-z-y))\,\mathrm{d}\mu_\mathfrak{z}(Z) = \int_\mathfrak{z}q((-x)(-z-y))\varphi(x,z) \,\mathrm{d}\mu_\mathfrak{z}(z) = 0
	\end{equation*}
	Because $[z\mapsto q((-x)(-z-y)]\in\mathcal P(\mathfrak{z})$. Hence $\mathcal T\varphi \in \mathcal S(G)\,\hat\otimes\,\mathcal S_*(G)$. Analogously we may prove that $\mathcal T^{-1}$ maps $\mathcal S(G)\,\hat\otimes\,\mathcal S_*(G)$ onto itself. Because $\mathcal S(G)\,\hat\otimes\,\mathcal S_*(G)$ carries the subspace topology in $\mathcal S(G)\,\hat\otimes\,\mathcal S(G)$, the continuity of $\mathcal T$ and $\mathcal T^{-1}$ on $\mathcal S(G)\,\hat\otimes\,\mathcal S_*(G)$ is evident. Since also
	\begin{equation*}
		\int_{G\times G} \psi\,\mathcal T\varphi\,\mathrm{d}(\mu\otimes\mu) = \int_{G\times G} \varphi\, \mathcal T^{-1}\psi \,\mathrm{d} (\mu\otimes\mu),
	\end{equation*}
	for all $\varphi,\psi\in\mathcal S(G\times G)$, we may extend $\mathcal T\restriction_{\mathcal S(G\times G)}$ to a Gelfand triple isomorphism.
\end{proof}

Now a direct conclusion is the formulation of the Kohn-Nirenberg quantization as a Gelfand triple isomorphism that incorporates the new Gelfand triples $\mathcal G_*(G,\mu)$ and $\mathcal G(\mathbb{R}^\times;\pi)$.
\begin{proposition}
	The \emph{Kohn-Nirenberg quantization in $\pi$-picture} 
	\begin{equation*}
		\operatorname{Op}_\pi:=\mathcal K^{-1}\mathcal T^{-1} (1\otimes\mathcal F^{-1}_\pi) \colon \mathcal G(G)\otimes \mathcal G(\mathbb{R}^\times;\pi) \to \mathcal L(\mathcal G_*(G),\mathcal G(G)),
	\end{equation*}
	where $\mathcal K$ is the kernel map between $\mathcal G(G)\otimes \mathcal G_*(G)$ and $\mathcal L(\mathcal G_*(G),\mathcal G(G))$, is a Gelfand triple isomorphism.
\end{proposition}
As for the Fourier transformation in $\pi$-picture, we may relate $\operatorname{Op}_\pi$ to the original Kohn-Nirenberg quantization $\operatorname{Op}$ via the diagrams on page \pageref{diagram:fourier_transforms}.

\subsection{The integral formula}

Representation in $\mathrm{SI/Z}_\mathbb{R}(\pi)$ can also be seen as slowly increasing functions. This is integral to our approach and will be proven in the proposition following the next lemma.

\begin{lemma}\label{lemma:GOM_2_GRtimesOM}
	Suppose $E$ is a complete locally convex space and $f\in\mathcal O_\mathrm{M}(G; E)$ and let $F(\lambda,x):=f(\delta_\lambda x)$. Then $F\in\mathcal O_\mathrm{M}(\mathbb R^\pm\times G;E)$.
\end{lemma}
\begin{proof}
	It is enough to show that for each continuous seminorm $p$ on $E$, each $k\in\mathbb{N}_0$ each $P\in\mathrm{Diff}_\mathcal{P}(G)$ there is a polynomial $q\in\mathcal P(G)$ and $l>0$, for which
	\begin{equation*}
		p(\partial_\lambda^k P_x F(\lambda,x))\le (1+ |\lambda|^l +|\lambda|^{-l})q(x).
	\end{equation*}
	We realize that there are polynomial differential operators $P_v$, such that
	\begin{equation*}
		\partial_\lambda^k P_x F(\lambda,x) = \sum_{v\in\mathbb{R}} \lambda^v (P_v f)(\delta_\lambda x),
	\end{equation*}
	as a finite linear combination. Since each $p(P_v f)$ is bounded by a polynomial $\widetilde q_v$, we may find polynomials $q_v$ such that
	\begin{equation*}
		p(\partial_\lambda^k P_x F(\lambda,x))\le \sum_{v\in\mathbb{R}}|\lambda|^v\widetilde q_v(\delta_\lambda x)=\sum_{v\in\mathbb{R}}|\lambda|^v q_v(x).
	\end{equation*}
	This concludes the proof.
\end{proof}
\begin{proposition}\label{prop:representation_is_op_valued_slowly_increasing}
	Suppose $\pi\in\mathrm{SI/Z}_\mathbb{R}(G)$, then the operator valued function $(x,\lambda)\mapsto \pi_\lambda(x)$ is both in $\mathcal O_{\rm M}(\mathbb{R}^\times\times G;\mathcal L(\mathcal{H}_\pi^\infty))$ and in $\mathcal O_{\rm M}(\mathbb{R}^\times\times G;\mathcal L(\mathcal{H}_\pi^{-\infty}))$
\end{proposition}
\begin{proof}	
	By Lemma \ref{lemma:GOM_2_GRtimesOM} it is enough to show, that $x\mapsto \pi(x)$ is slowly increasing. For this purpose we choose an equivalent representation, that is more easily understood. There is a representation $\sigma\sim\pi$ on $\mathcal H_\sigma = L^2(\mathbb{R}^n)$, such that $\mathcal H_\sigma^\infty = \mathcal S(\mathbb{R}^n)$ and
	\begin{equation*}
		\sigma(x)f(t) = \mathrm{e}^{2\pi\mathrm{i} \xi(a(x,t))}f(x^{-1}\cdot t)
	\end{equation*}
	where $\xi$ is a linear functional on a subalgebra $\mathfrak{m}$ of $\mathfrak{g}$, $a\colon G\times \mathbb{R}^n\to\mathfrak{m}$ is polynomial and $G\times \mathbb{R}^n \ni (x,t)\mapsto x\cdot t\in\mathbb{R}^n$ is a polynomial action of $G$ on $\mathbb{R}^n$ by \cite{pedersen_geometric_quant} and \cite[Corollary 4.1.2]{rep_nilpotent_lie_groups}. Because $(x,t)\mapsto x^{-1}\cdot t$ is polynomial, we may represent the action of $G$ on $\mathbb{R}^n$ by a linear combination
	\begin{equation*}
		x\cdot t=\sum_{j,k} s_{k,j}(x)\,u_{k,j}(t)\, e_j,
	\end{equation*}
	where $(e_j)_j$ is the standard basis on $\mathbb{R}^n$ and $ s_{k,j}$, $u_{k,j}$ are polynomials. Thus, we also have
	\begin{equation*}
		t_j\, \sigma(x)f(t) = \sum_{k}s_{k,j}(x)\,\sigma(x)(u_{k,j}\, f)(t).
	\end{equation*}
	For the same reason, there are polynomials $q_{j,k}$, $\widetilde q_{j,k}$ on $G$, $r_{j,k}$, $\widetilde r_{j,k}$ on $\mathbb{R}^n$ such that
	\begin{equation*}
	\begin{split}
		\partial_{t_j} f(x^{-1}\cdot t) &= \sum_{k} \widetilde q_{j,k}(x)\, \widetilde r_{j,k}(t)\, (\partial_k f)(x^{-1}\cdot t)\\
		&= \sum_{k} q_{j,k}(x)\, r_{j,k}(x^{-1 }\cdot t)\, (\partial_k f)(x^{-1}\cdot t).
	\end{split}
	\end{equation*}
	Hence, for all $\alpha,\beta\in\mathbb{N}_0$ we find operators $A_k\in\mathcal L(\mathcal S(\mathbb{R}^n))$ and polynomials $v_k\in\mathcal P(G)$, such that
	\begin{equation*}
		t^\beta\partial_t^\alpha \sigma(x)f(t) = \sum_k v_k(x)\,\sigma(x)(A_kf)(t),
	\end{equation*}
	as a finite linear combination.
	
	The topology on $\mathcal L(\mathcal S(\mathbb{R}^n))$ is induced by the seminorms
	\begin{equation*}
		p\colon A\mapsto \sup_{f\in B}\sup_{t\in\mathbb{R}^n} | t^\beta\partial^\alpha_t Af(t)|,\quad B\subset \mathcal S(\mathbb{R}^n)\text{ bounded}, \alpha,\beta\in\mathbb{N}_0^n.
	\end{equation*}
	Now if $L\in \mathfrak u(\mathfrak{g})_\mathrm{L}$ is any left invariant differential operator on $G$ and $p$ is a seminorm as above, we get
	\begin{equation*}
	\begin{split}
		p(L_x \sigma(x)) &\le \sum_k v_k(x) \sup_{f\in B}\sup_{t\in\mathbb{R}^n} |\sigma(x)(A_k\sigma(L)f)(t)| \\ &= \sum_k v_k(x) \sup_{f\in B}\sup_{t\in\mathbb{R}^n} |(A_k\sigma(L)f)(t)|.
	\end{split}
	\end{equation*}
	The right-hand side of the above inequality is a sum of continuous seminorms times polynomials, since $\sigma(L)\in\mathcal L(\mathcal S(\mathbb{R}^n))$. Thus $x\mapsto \sigma(x)$ is slowly increasing. Due to $\pi\sim\sigma$ the map $x\mapsto \pi(x)$ is slowly increasing, too. Now $(x,\lambda)\mapsto \pi_\lambda$ is slowly increasing with values in $\mathcal L(H^\infty_\pi)$ due to Lemma \ref{lemma:GOM_2_GRtimesOM}. We finish the proof by remarking that $\mathcal L(H^\infty_\pi)$ and $\mathcal L(H^{-\infty}_\pi)$ are isomorphic by the transposition and $\pi_\lambda(x)^\mathrm{t} = \mathcal C_\pi\pi_\lambda(-x)\mathcal C_\pi$. This implies that $\pi$ is also slowly increasing with values in $\mathcal L(H^{-\infty}_\pi)$.
\end{proof}

With the help of the above proposition, we want to write the inverse Fourier transform as an integral, which converges in $\mathcal O_\mathrm{M}(G)$. For this purpose, we need to explain a small fact about the dual space $\mathcal O_\mathrm{M}'(G)$. Denote by $\partial_1,\partial_2,\dots$ the directional derivative to any basis $v_1,v_2,\dots$ of $\mathfrak{g}$. Each continuous linear functional on $\mathcal O_\mathrm{M}(\mathfrak{g})$ can be represented by the set
\begin{equation}\label{eq:dual_of_OM_char}
	\mathcal O_\mathrm{M}'(G) = \operatorname{span}_\mathbb{C}\{\partial^\alpha f \mid \alpha\in\mathbb{N}^{\dim(G)}_0\text{ and }f\in C(G) \text{ is rapidly decreasing}\},
\end{equation}
where we used the standard multi-index notation, see \cite[page 130 of chapter 2]{grothendieck_tensorprodukte}, if we use the dual pairing
\begin{equation*}
	\langle \partial^\alpha f, g\rangle := \int_G f\, (-\partial)^\alpha g\,\mathrm{d}\mu.
\end{equation*}
Here we say $f\colon G\to \mathbb{C}$ is rapidly decreasing, iff $q f$ is a bounded function for any $q\in\mathcal P(G)$. The differential operators $\partial^\alpha$, $\alpha\in\mathbb{N}^{\dim(G)}_0$ span the $\mathcal P(G)$-module $\mathrm{Diff}_\mathcal{P}(G)$. Since the multiplication of Schwartz functions with polynomials is continuous, we may exchange $\partial^\alpha$ with arbitrary $P\in\mathrm{Diff}_\mathcal{P}(G)$ in the pairing above. By \cite[Lemma A.2.2]{rep_nilpotent_lie_groups} the $\mathcal P(G)$-span of the left invariant differential operators $\mathfrak{u}(\mathfrak{g}_\mathrm{L})$ is equal to $\mathrm{Diff}_\mathcal{P}(G)$. Now let $w^1,w^2,\dots$ be the dual basis to $v_1,v_2,\dots$ and let $X_1, X_2,\dots$ be the left invariant vector fields associated to $v_1,v_2,\dots$. A quick calculation shows that for all $\phi\in\mathcal S'(G)$ and all $j,k$ there exists a polynomial $q\in\mathcal P(G)$ with
\begin{equation*}
	w^j\, X_k\phi = q\,\phi + X_k(w^j\, \phi).
\end{equation*}
Of course, the set of rapidly decreasing continuous functions is invariant under the multiplication with polynomials. In conclusion, we may represent the dual to $\mathcal O_\mathrm{M}(G)$ by
\begin{equation*}
	\mathcal O_\mathrm{M}'(G) = \operatorname{span}_\mathbb{C} \{Pf \mid P\in\mathfrak{u}(\mathfrak{g}_\mathrm{L})\text{ and }f\in C(G) \text{ is rapidly decreasing}\}.
\end{equation*}

\begin{lemma}\label{lemma:SFourier_conv_in_OM}
	If $\varphi\in\mathcal S(\mathfrak{g})$ and $\omega^\times\ni\ell\sim\pi\in\mathrm{SI/Z}(G)$, then the integral
	\begin{equation*}
		\varphi = \int_{\mathbb R^\times} \operatorname{Tr}[\pi_\lambda\,\mathcal F_\pi\varphi(\lambda)]\,\mathrm d\lambda_\pi
	\end{equation*}
	exists in $\mathcal O_\mathrm{M}(G)$, where $\,\mathrm{d}\lambda_\pi := \kappa\,|P\!\! f(\ell)|\,|\lambda|^{Q-1}\,\mathrm{d}\lambda$.
\end{lemma}
\begin{proof}
	Let $f\colon G\to\mathbb{C}$ be continuous and rapidly decreasing, let $P\in\mathfrak{u}(\mathfrak{g}_\mathrm{L})$ and let $\varphi\in\mathcal S(G)$. Then, $f$ and $P^\mathrm{t}\varphi$ are $L^2$ functions and we may apply Plancherel for $\mathcal F_\pi$. Hence
	\begin{equation*}
		\langle P^\mathrm{t} \overline f,\varphi\rangle = \int_G \overline f\, P\varphi\,\mathrm{d}\mu = \int_{\mathbb R^\times } \operatorname{Tr}[\widehat f(\lambda)^* \, \mathcal F_\pi(P\varphi)(\lambda)]\,\mathrm d \lambda_\pi,
	\end{equation*}
	in which we used the shorthand $\widehat g(\lambda) = \mathcal F_\pi g(\lambda)$ for functions $g$. 
	Since $f\in L^1(G,\mu)$, we know that the integral that evaluates the Fourier transform in $\pi$-picture converges in $\mathcal L(H_\pi)$ with respect to the weak operator topology. That means for each pair $u,v\in H_\pi$ we have
	\begin{equation*}
		\boldsymbol{(}\widehat f(\lambda)^*u,v\boldsymbol{)}_{H_\pi} = \int_G \overline{f(x)}\, \boldsymbol{(}\pi_\lambda(x)u,v\boldsymbol{)}_{H_\pi}\,\mathrm{d}\mu(x).
	\end{equation*}
	Because $P\varphi\in\mathcal S(G)$, we have $\mathcal F_\pi(P\varphi)(\lambda)= \pi_\lambda(P)\,\widehat \varphi(\lambda)\in\mathcal L(H^{-\infty}_\pi,H^\infty_\pi)$, which is  a nuclear operator for each $\lambda\in\mathbb R^\times$. Hence for each orthonormal basis $(e_k)_{k\in\mathbb N}\subset H_\pi$
	\begin{align*}
		\int_G \sum_{k\in\mathbb N} |\overline{f(x)}\,\boldsymbol{(}\pi_\lambda(x)\,\pi_\lambda(P)\,\widehat \varphi(\lambda)e_k,e_k&\boldsymbol{)}_{H_\pi}|\,\mathrm{d}\mu(x)\\
		&\le \|f\|_{L^1(G,\mu)}\, \|\pi_\lambda(P)\,\widehat\varphi(\lambda)\|_{\mathcal N(H_\pi)}<\infty,
	\end{align*}
	where $\|\cdot\|_{\mathcal N(H_\pi)}$ is the trace-norm on the space of nuclear operators on $H_\pi$. Using Fubini with respect to the counting measure and $\mu$ results in
	\begin{equation*}
		\operatorname{Tr}[\widehat f(\lambda)^* \, \mathcal F_\pi(P\varphi)(\lambda)] =\int_G \overline{f(x)}\operatorname{Tr}[\pi_\lambda(x)\, \pi_\lambda(P)\,\widehat \varphi(\lambda)]\,\mathrm{d}\mu(x), 
	\end{equation*}
	since $f\in L^1(G,\mu)$. Naturally, we have $\pi_\lambda(x)\, \pi_\lambda(P) = P_x\pi_\lambda(x)$. By the embedding of $\mathcal L(H^{-\infty}_\pi,H^\infty_\pi)$ into the nuclear operators $\mathcal N(H_\pi)$, we may see $\operatorname{Tr}$ as a continuous functional on $\mathcal L(H^{-\infty}_\pi,H^\infty_\pi)$. Because the operator valued function $\pi_\lambda\,\widehat\varphi(\lambda)$ is a slowly increasing map from $G$ to $\mathcal L(H^{-\infty}_\pi,H^\infty_\pi)$, we get
	\begin{align*}
		\operatorname{Tr}[\pi_\lambda(x)\, \pi_\lambda(P) \,\widehat \varphi(\lambda)] &= P_x \operatorname{Tr}[\pi_\lambda(x) \,\widehat \varphi(\lambda)],\\ 
		\operatorname{Tr}[\pi_\lambda \,\widehat \varphi(\lambda)]&\in  \mathcal O_\mathrm{M}(G).
	\end{align*}
	Finally we get
	\begin{align*}
		\langle P^\mathrm{t} \overline f,\varphi\rangle &= \int_{\mathbb R^\times }\int_G \overline{f(x)}\,P_x\operatorname{Tr}[\pi_\lambda(x)\,\widehat \varphi(\lambda)]\,\mathrm d\mu(x)\mathrm d\lambda_\pi \\
		&= \int_{\mathbb R^\times }\langle P^\mathrm{t}\overline{f},\operatorname{Tr}[\pi_\lambda(\cdot)\,\widehat \varphi(\lambda)]\rangle\,\mathrm d\lambda_\pi,
	\end{align*}
	which concludes the proof.
\end{proof}

Let us write $\rho(x,\lambda):=\pi_\lambda(x)$ and $\rho^*(\lambda,x):=\pi_\lambda(-x)$ for some $\pi\in\mathrm{SI/Z}(G)$. With Lemma \ref{lemma:allg_tenserop_stetig_alt}, we already proved the continuity of the map
\begin{equation*}
	 \mathcal L(\mathcal O_\mathrm{M}(G))\to \mathcal L(\mathcal O_\mathrm{M}(G\times\mathbb R^\times;\mathcal L(H^\infty_\pi))),\ A\mapsto A\otimes 1.
\end{equation*}
Of course the evaluation map
\begin{equation*}
	\mathcal L(\mathcal O_\mathrm{M}(G\times\mathbb R^\times;\mathcal L(H^\infty_\pi))) \to \mathcal O_\mathrm{M}(G\times\mathbb R^\times;\mathcal L(H^\infty_\pi)),\ \Phi \mapsto \Phi(\rho)
\end{equation*}
is continuous, as well. Finally, since the multiplication in $\mathcal O_\mathrm{M}(G\times\mathbb R^\times)$ is continuous \cite[page 248]{Schwartz_distributions} and because of Theorem \ref{theorem:allg_fortsetzung_bilinear}, the map $S$ defined by
\begin{equation*}
	S\colon \mathcal L(\mathcal O_\mathrm{M}(G))\to \mathcal O_\mathrm{M}(G\times\mathbb R^\times;\mathcal L(H^\infty_\pi)),\ A\mapsto \rho^*\cdot(A\otimes 1)(\rho)
\end{equation*}
is continuous. Now this map looks exactly like the inverse Kohn-Nirenberg quantization on compact Lie groups $H$ from  \cite{turunen_ruzhansky_psydos_and_symmetries}. Namely, for any $B\in\mathcal L(\mathcal D(H))$ the unique Kohn-Nirenberg symbol $b$ with $B=\operatorname{Op}(b)$, evaluated at the irreducible unitary representation $\xi$, is given by $\xi^*\cdot (A\otimes 1)(\xi)\in\mathcal D(H;\mathcal L( H_\xi))$.

\begin{lemma}\label{lemma:OM_densein_G}
	The embedding $\mathcal S_*(G)\hookrightarrow \mathcal{O}_\mathrm{M}(G)$ is continuous and has dense range.
\end{lemma}
\begin{proof}
	The multiplication on $\mathcal S(G)$ is a continuous bilinear map. This implies the continuity of the canonical embedding $\tau\colon \mathcal S_*(G)\hookrightarrow \mathcal O_\mathrm{M}(G)$, since $\mathcal S_*(G)$ carries the subspace topology in $\mathcal S(G)$. Now consider the dual map
	\begin{equation*}
		\tau' \colon \mathcal O_{\rm M}'(G) \to \mathcal S_*'(G),\quad\text{where}\quad 
		\langle\tau'\phi,\varphi\rangle=\langle \phi,\varphi\rangle, \quad\text{for all }\varphi\in\mathcal S_*(G).
	\end{equation*}
	That this is indeed an embedding, can be seen from Proposition \ref{prop:distr_conv_to_zero_embedding} and the representation \eqref{eq:dual_of_OM_char} of the dual space $\mathcal O_\mathrm{M}'(G)$. By the Hahn-Banach theorem, the operator $\tau$ has dense image.
\end{proof}

In the Lemma above we saw that $\mathcal S_*(G)\hookrightarrow \mathcal O_\mathrm{M}(G)$ has dense range. Naturally we also have $\mathcal O_\mathrm{M}(G)\hookrightarrow \mathcal S'(G)$ and $\mathcal O_\mathrm{M}(\mathbb R^\times)\hookrightarrow \mathcal S'(\mathbb R^\times)$, thus we get embeddings
\begin{eqnarray*}
	\mathcal L(\mathcal O_\mathrm{M}(G))&\hookrightarrow& \mathcal L(\mathcal S_*(G),\mathcal S'(G)),\\
	\mathcal O_\mathrm{M}(G\times\mathbb R^\times;\mathcal L(H^\infty_\pi))&\hookrightarrow& \mathcal S'(G;\mathcal S'(\mathbb R^\times;\pi)).
\end{eqnarray*}
Notice that we can exchange $\mathcal L(H^\infty_\pi)$ with $\mathcal L(H^{-\infty}_\pi)$, in the paragraph above. By using the embeddings above, we will see that the map $S$ does indeed reproduce the Kohn-Nirenberg symbol. We can even go one step further. Of course for $A\in\mathcal L(\mathcal O_\mathrm{M}(G),\mathcal S(G))$, we can still define the map $S$, since $\mathcal S(G)\hookrightarrow\mathcal O_\mathrm{M}(G)$. However, we are lacking tools to check whether $S(A)\in \mathcal S(G)\,\hat\otimes\,\mathcal O_\mathrm{M}(\mathbb R^\times;\mathcal L(H^\infty_\pi))$ or not, since we cannot apply Theorem \ref{theorem:allg_fortsetzung_bilinear}.
We run into the same problem if we try to define $S$ for operators $A\in\mathcal L(\mathcal O_\mathrm{M}(G);\mathcal S'(G))$.

Before we prove that the definition of $S$ gives us Kohn-Nirenberg symbols, we need two final lemmata.	 

\begin{lemma}\label{lemma:multiplication_symbol_rep_formula}
	Suppose $a\in\mathcal S'(G)\,\hat\otimes\,\mathcal S'(\mathbb R^\times;\pi)$, then
	\begin{equation*}
		\rho\cdot a = (1\otimes\mathcal F_\pi\operatorname{inv})\mathcal T^{-1}(1\otimes \mathcal F_\pi^{-1})a,
	\end{equation*}
	where $\operatorname{inv}f(x)=f(-x)$, $f\in\mathcal S(G)$, continued to distributions.
\end{lemma}
\begin{proof}
	First we take $a\in\mathcal S(G)\,\hat\otimes\,\mathcal S(\mathbb R^\times;\pi)$. Then we just have
	\begin{equation*}
		(1\otimes \mathcal F_\pi^{-1})(\rho\cdot a)(x,y) = (1\otimes\mathcal F_\pi^{-1})a(x,yx) = (1\otimes \operatorname{inv})\mathcal T^{-1}(1\otimes \mathcal F_G^{-1})a(x,y),
	\end{equation*}
	by the integral formula for the inverse Fourier transform from Lemma \ref{lemma:SFourier_conv_in_OM}. Now the rest simply follows due to the continuity of the involved maps.
\end{proof}

\begin{lemma}\label{lemma:formula_character_2_SIrep}
	Define $\chi_x(\xi):=\mathrm{e}^{2\pi\mathrm{i} \xi(x)}$ for $x\in\mathfrak{g}$ and $\xi\in\mathfrak{g}'$. Then
	\begin{equation*}
		\mathfrak{Op}_\pi\wp_\ell(\chi_x) = \pi_\lambda(x)
	\end{equation*}
	for any $\lambda\in\mathbb{R}^\times$ and $\mathrm{SI/Z}_\mathbb{R}(G)\ni\pi\sim \ell\in\omega^\times$.
\end{lemma}
\begin{proof}
	Let $\psi\in\mathcal D(\mathfrak{z}^\circ)$, such that $\psi\equiv 1$ on some neighbourhood of zero. If we define $\psi_k(x):=\psi(x/k)$ for $k\in\mathbb{N}$, then $\psi_k \chi_x\to \chi_x$ for $k\to\infty$ in $\mathcal S'(\mathfrak{z}^\circ)$. Due to the continuity of $\mathfrak{op}_\pi$, we may deduce for $\lambda>0$
	\begin{equation*}
	\begin{split}
		\mathfrak{Op}_\pi\wp_\ell(\chi_x) &= \lim_{k\to \infty} \mathrm{e}^{2\pi\mathrm{i}\ell(\delta_\lambda x)} \mathfrak{op}_\pi(\psi_k\cdot {\chi_{\delta_\lambda x}}\restriction_{\mathfrak{z}^\circ}) \\
		&= \lim_{k\to \infty} \mathrm{e}^{2\pi\mathrm{i}\ell(\delta_\lambda x)} \int_\omega \pi(y) \widehat \psi_k(y-\delta_\lambda \widetilde x)\,\mathrm{d}\nu(y),
		\end{split}
	\end{equation*}
	in which $\widehat \psi_k\in\mathcal S(\omega)$ is the euclidean Fourier transform of $\psi_k$ and $\widetilde x$ is the projection of $x$ onto $\omega$ along $\mathfrak{z}$. If we consider the functions $\widehat \psi_k(\cdot -\delta_\lambda\widetilde x)$ as distributions
	\begin{equation*}
		\mathcal S(G)\ni \varphi \mapsto\int_\omega  \widehat \psi_k(y -\delta_\lambda\widetilde x)\varphi(y)\,\mathrm{d}\nu(y)
	\end{equation*}
	in $\mathcal S'(G)$, then the sequence of functions $\widehat \psi_k(\cdot -\delta_\lambda\widetilde x)$ converges to the Dirac distribution supported on $\delta_\lambda x$ in $\mathcal S'(G)$. By Proposition \ref{prop:distr_conv_to_zero_embedding} and the continuity of $\mathcal F_\pi$, we arrive at
	\begin{equation*}
		\mathfrak{Op}_\pi\wp_\ell(\chi_x) = \mathrm{e}^{2\pi\mathrm{i} \ell(\delta_\lambda x)} \pi (\delta_\lambda\widetilde x) =\pi_\lambda(x).
	\end{equation*}
	For $\lambda<0$ the calculation is analogous, we merely need to exchange $\pi$ with $\overline{\pi}$.
\end{proof}

\begin{theorem}
	For any $A\in\mathcal L(\mathcal O_\mathrm{M}(G),E)$, $E\in\{\mathcal S(G),\mathcal O_\mathrm{M}(G)\}$, the equality $a:=S(A)=\operatorname{Op}_\pi^{-1}(A)$ is valid. Furthermore,
	\begin{equation*}
		A\,\varphi = \int_{\mathbb R^\times} \operatorname{Tr}[\pi_\lambda\, a(\cdot,\lambda) \,\mathcal F_\pi\varphi(\lambda)]\,\mathrm{d}\lambda_\pi,\quad\text{for}\quad \varphi\in\mathcal S(G),
	\end{equation*}
	where the integral exists in $E$.
\end{theorem}
\begin{proof}
	First we will prove the integral formula for $A\in\mathcal L(\mathcal O_\mathrm{M}(G),E)$. From Lemma \ref{lemma:SFourier_conv_in_OM} we know, that for $\varphi\in\mathcal S(G)$
	\begin{equation*}
		A\, \varphi = A\,\int_G \operatorname{Tr}[\pi_\lambda\, \widehat \varphi(\lambda)]\,\mathrm{d}\lambda_\pi = \int_G A\big(\operatorname{Tr}[\pi_\lambda\, \widehat \varphi(\lambda)]\big)\,\mathrm{d}\lambda_\pi,
	\end{equation*}
	where the integral converges in $E$ and where we used the shorthand $\mathcal F_\pi\varphi(\lambda)=\widehat \varphi (\lambda)$. Due to Proposition \ref{proposition:schwartz_fourier_picture} and Proposition \ref{prop:representation_is_op_valued_slowly_increasing}, we know that
	\begin{equation*}
		(\pi_\lambda\, \widehat\varphi(\lambda))\in \mathcal O_\mathrm{M}(G)\,\hat\otimes\, \mathcal L(H^{-\infty}_\pi, H^\infty_\pi).
	\end{equation*}
	The trace operator $\operatorname{Tr}$, restricted from the nuclear operators on $H_\pi$, is a continuous functional on $\mathcal L(H^{-\infty}_\pi, H^\infty_\pi)$, so we may use the tensor product structure of the above expression to get
	\begin{equation*}
		A\big(\operatorname{Tr}[\pi_\lambda\, \widehat \varphi(\lambda)]\big) = (A\otimes\operatorname{Tr})\big(\pi_\lambda\, \widehat \varphi(\lambda)\big) = (1\otimes\operatorname{Tr})(A\otimes 1)\big(\pi_\lambda\, \widehat \varphi(\lambda)\big),
	\end{equation*}
	for each $\lambda\in\mathbb R^\times$. Furthermore,
	\begin{equation*}
		(A\otimes 1)(\pi_\lambda\, \widehat\varphi(\lambda)) = \pi_\lambda\cdot\pi_\lambda^*\cdot (A\otimes 1)(\pi_\lambda)\cdot \widehat\varphi(\lambda),
	\end{equation*}
	in which the multiplication is defined pointwise by the multiplication in $\mathcal L(H^\infty_\pi)$. Hence, we can represent $A\,\varphi$ by the integral
	\begin{equation*}
		A\,\varphi = \int_{\mathbb R^\times}\operatorname{Tr}[\pi_\lambda\,a(\cdot,\lambda)\,\widehat \varphi(\lambda)]\,\mathrm{d}\lambda_\pi,
	\end{equation*}
	with $a:=S(A)$.
	
	Now it is left to check that indeed $A = \operatorname{Op}_\pi(a)$. First of all, due to Lemma \ref{lemma:multiplication_symbol_rep_formula}
	\begin{equation*}
		\mathcal T^{-1}(1\otimes\mathcal F_\pi^{-1})a = (1\otimes\operatorname{inv}\mathcal F_\pi^{-1})(\rho\cdot a).
	\end{equation*}
	We define the function $\chi(x,\xi):=\mathrm{e}^{2\pi\mathrm i \xi(x)}$ for $\xi\in\mathfrak{g}'$, $x\in\mathfrak{g}$, then $\chi\in\mathcal O_\mathrm{M}(\mathfrak{g}\times\mathfrak{g}^\times)$. Because $(1\otimes\mathfrak{Op}_\pi \wp_\ell)\chi(x,\lambda) = \pi_\lambda(x) = \rho(x,\lambda)$, due to Lemma \ref{lemma:formula_character_2_SIrep}, and $\mathcal F_\pi = \mathfrak{Op}_\pi\wp_\ell \mathcal F_\mathfrak{g}$, we know that
	\begin{equation*}
		(1\otimes \operatorname{inv}\mathcal F_\pi^{-1})(A\otimes 1)(\rho) = (A\otimes \operatorname{inv}\mathcal F_\mathfrak{g}^{-1})(\chi) = (A\otimes\mathcal F_{\mathfrak{g'}})(\chi).
	\end{equation*}
	We choose arbitrary $\varphi\in\mathcal S(\mathfrak{g})$ and $\psi\in\mathcal S_*(G)$. The integral
	\begin{equation*}
		\varphi = \int_\mathfrak{g'} \chi(\cdot,\xi)\,\mathcal F_\mathfrak{g}\varphi(\xi)\,\mathrm{d}\mu'(\xi)
	\end{equation*}
	converges in $\mathcal O_\mathrm{M}(\mathfrak{g})$. Hence
	\begin{eqnarray*}
		\langle (A\otimes\mathcal F_{\mathfrak{g}'})\chi,\psi\otimes\varphi\rangle &=& \langle (A\otimes 1)\chi, \psi\otimes\mathcal F_\mathfrak{g}\varphi\rangle \\
		&=& \int_{\mathfrak g'} \langle A(\chi(\cdot,\xi)),\psi\rangle\,\mathcal F_\mathfrak{g}\varphi(\xi)\,\mathrm{d}\mu'(\xi) \\
		&=& \langle A\varphi,\psi\rangle.
	\end{eqnarray*}
	Combining the calculations above implies
	\begin{equation*}
		\mathcal K A = \mathcal T^{-1}(1\otimes\mathcal F_\pi^{-1})a,
	\end{equation*}
	for the kernel map $\mathcal K$ for $\mathcal G(G)\otimes\mathcal G_*(G)$. I.e. $\operatorname{Op}^{-1}(A) = S(A)$.
\end{proof}

\begin{acknowledgement} 
Jonas Brinker was supported by ISAAC as a young scientist for his participation at the 12th ISAAC Congress in Aveiro.

We thank David Rottensteiner for discussing the possibility of our approach in this setting (homogeneous Lie groups with one-dimensional center and flat orbits) and we thank Christian Bargetz for a very helpful discussion and for pointing out the relevant literature to the multiplication of vector valued distributions with vector valued functions. 

We especially thank the reviewer for many helpful comments and suggestions, which allowed to improve the paper.
\end{acknowledgement}


\end{document}